\newif\ifJOURNAL\global\JOURNALfalse
\newif\ifHYPER\global\HYPERtrue
\newif\ifINTERNAL\global\INTERNALfalse
\ifJOURNAL
\documentclass{ieja}
\usepackage{amsmath,amsthm,amssymb,amsfonts,amscd,array}
\usepackage[all]{xy}

\else
\documentclass[10pt,b5paper,fleqn]{article}
\usepackage{amsmath,amsfonts,amssymb}
\usepackage{amsthm}
\fi
\usepackage[utf8]{inputenc}
\usepackage[mathscr]{eucal}
\usepackage{color}
\usepackage{graphicx}

\ifHYPER
\definecolor{ks-green}{rgb}{0.0,0.7,0.0}
\definecolor{ks-red}{rgb}{0.7,0.0,0.0}
\definecolor{ks-blue}{rgb}{0.0,0.0,0.7}
\usepackage{hyperref}
\hypersetup{colorlinks=true,citecolor=ks-green,linkcolor=ks-red}
\else
\newcommand{\url}[1]{\texttt{#1}}
\fi

\ifJOURNAL

\newtheorem{theorem}{Theorem}[section]
\newtheorem{corollary}[theorem]{Corollary}
\newtheorem{lemma}[theorem]{Lemma}
\newtheorem{proposition}[theorem]{Proposition}
\newtheorem{definition}[theorem]{Definition}
\newtheorem{Remark}[theorem]{Remark}
\newtheorem{remark}[theorem]{Remark}
\newtheorem{Example}[theorem]{Example}
\newtheorem{example}[theorem]{Example}
\newtheorem{notation}[theorem]{Notation}
\else
\numberwithin{equation}{section}
\theoremstyle{plain}
\newtheorem{theorem}[equation]{Theorem}
\newtheorem{lemma}[equation]{Lemma}
\newtheorem{corollary}[equation]{Corollary}
\newtheorem{proposition}[equation]{Proposition}

\theoremstyle{definition}
\newtheorem{definition}[equation]{Definition}
\newtheorem{Example}[equation]{Example}
\newtheorem{Remark}[equation]{Remark}
\newenvironment{example}{\emph{Example.}}{}
\newenvironment{remark}{\emph{Remark.}}{}
\newenvironment{notation}{\emph{Notation.}}{}
\fi

\newcommand{\block}[1]{\underline{#1}}
\newcommand{\tsfrac}[2]{\textstyle{\frac{#1}{#2}}}
\newcommand{\tabstrut}{\rule[-0.5ex]{0pt}{2.8ex}}

\newcommand{\ldivs}{\!\mid_{\mkern-1mu\text{l}}\!}
\newcommand{\rdivs}{\!\mid_{\mkern-1mu\text{r}}\!}

\newcommand{\nldivs}{\!\nmid_{\mkern-1mu\text{l}}\!}

\newcommand{\Fbullet}{\mathbb{H}}
\newcommand{\atoms}{\mathbf{A}}

\def\trp{^{\!\top}}
\def\inv{^{-1}}

\def\numN{\mathbb{N}}

\def\numQ{\mathbb{Q}}

\def\numC{\mathbb{C}}

\newcommand{\ncRATS}[2]{#1^{\text{rat}}\langle\!\langle #2\rangle\!\rangle}
\newcommand{\freeALG}[2]{#1\langle #2\rangle}
\newcommand{\freeFLD}[2]{#1(\!\langle #2\rangle\!)}
\newcommand{\perm}{\Sigma}

\DeclareMathOperator{\height}{ht}

\DeclareMathOperator{\rank}{rank}

\DeclareMathOperator{\linsp}{span}

\newcommand{\field}[1]{\mathbb{#1}}
\newcommand{\als}[1]{\mathcal{#1}}
\newcommand{\ideal}[1]{\mathfrak{#1}}
\newcommand{\aclo}[1]{\overline{#1}}
\newcommand{\length}[1]{\vert #1 \vert}

\newcommand{\Fldivs}{\!\mid^{\field{F}}_{\mkern-1mu\text{l}}\!}
\newcommand{\Frdivs}{\!\mid^{\field{F}}_{\mkern-1mu\text{r}}\!}

\ifJOURNAL
\title{A Factorization Theory for some Free Fields}
\author{Konrad Schrempf}
\address{Konrad Schrempf\\
Faculty of Mathematics\\
University of Vienna\\
Oskar-Morgenstern-Platz~1\\
1090 Wien, Austria \\
e-mail: math@versibilitas.at}
\thanks{I thank Daniel Smertnig for the fruitful discussions about
non-commutative factorization and Michael Moßhammer for some
hints on graphs and trees and use this opportunity to
thank Sergey Berezin and Vladimir Vasilchuk
for their support in St.~Petersburg in May 2017.
I am very grateful for the constructive feedback of the anonymous
referees to increase readability,
in particular for the suggested simplification of the definition
of left/right divisibility.}
\else
\title{A Factorization Theory\\
  for some Free Fields}
\author{Konrad Schrempf%
  \footnote{Contact: math@versibilitas.at (Konrad Schrempf),
    \url{https://orcid.org/0000-0001-8509-009X},
    Universität Wien, Fakultät für Mathematik,
    Oskar-Morgenstern-Platz~1, 1090 Wien, Austria.
    }
  \hspace{0.2em}\href{https://orcid.org/0000-0001-8509-009X}{%
  \includegraphics[height=10pt]{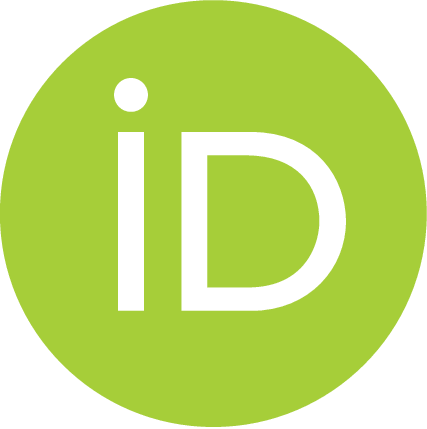}}
  }
\fi

\begin{document}

\maketitle

\begin{abstract}
Although in general there is no meaningful
concept of factorization in fields,
that in free associative algebras
(over a \emph{commutative} field)
can be extended to their respective free field
(universal field of fractions) on the level of
\emph{minimal linear representations}.
We establish a factorization theory by
providing an alternative definition of
left (and right) divisibility based
on the \emph{rank} of an element and
show that it coincides with the ``classical''
left (and right) divisibility for non-commutative
polynomials.
Additionally we present an approach to factorize elements,
in particular \emph{rational formal power series},
into their (generalized) atoms.
The problem is reduced to solving a system of
polynomial equations with \emph{commuting} unknowns.
\ifJOURNAL
\\[+2mm]
\subjclassname{ Primary 16K40, 16Z05; Secondary 16G99, 16S10}
\fi
\end{abstract}

\ifJOURNAL
\else
\medskip
\emph{Keywords and 2020 Mathematics Subject Classification.}
Free associative algebra,
factorization of non-commutative polynomials,
minimal linear representation, universal field of fractions,
admissible linear system, non-commutative formal power series;
Primary 16K40, 16Z05; Secondary 16G99, 16S10
\fi

\ifJOURNAL
\section{Introduction}
\else
\section*{Introduction}
\fi

From an algebraic point of view fields are usually
not very interesting (with respect to factorization)
due to the lack of ``structure'',
for example, they do not have \emph{non-zero} \emph{non-units}.
However here, the field
---the \emph{universal field of fractions} (``free field'')
of the \emph{free associative algebra} (over a \emph{commutative}
ground field)--- is \emph{non-commutative}
and \emph{infinite dimensional} over its center
(at least if we exclude the one-variable case).
A brief introduction can be found in
\cite[Secton~9.3]{Cohn2003b}
, for details we refer to
\cite[Chapter~7]{Cohn2006a}
, where also historical information is provided:
``Until 1970 the only purely algebraic methods of embedding rings in
fields were based on Ore's method 
\cite{Ore1931a}
.''

\medskip
The main idea is to view all elements in terms of their
\emph{normal form} (minimal linear representation)
\cite{Cohn1994a}
. Given an element, the dimension of a \emph{minimal}
linear representation defines its \emph{rank}
\cite{Cohn1999a}
, for example,
the rank of a word/monomial of length $n$ is $n+1$.
Since multiplication (of two elements) can be formulated
in terms of linear representations, we establish a
concept to ``reverse'' this step, that is, given an element
(by a \emph{minimal} linear representation) to find left
(and right) divisors subject to conditions on the ranks
of the involved elements.

In \cite{Schrempf2017b9}
\ we showed that in the \emph{free associative algebra}
there is a rather natural correspondence between
a factorization of an element and (upper right) blocks of zeros
in (a special form of) its minimal linear representations.
One does not have to take care about the ranks.
In general a \emph{minimal} multiplication, that is,
a multiplication on the level of \emph{minimal} linear representations
is much more subtle.
Now, \emph{how} do we have to define divisibility
in terms of the rank
such that it is equivalent to that in the free associative algebra?

\medskip
Joining factorization theory in the non-commutative
setting ---for an overview see
\cite{Smertnig2016a}
---
and the theory of embedding ``non-commutative'' rings
into a (skew) field
---to be more precise: embedding \emph{free ideal rings} (firs)
into their respective universal field of fractions,
see \cite[Chapter~2]{Cohn2006a}
---
even for the ``simplest'' case of the \emph{free associative algebra}
results in a very rich structure,
maybe not only for a ``free factorization theory''.
Somewhat paradoxical is the fact that the \emph{inverse}
plays a crucial role. Since \emph{each} non-zero element
(in the free field) is invertible, we can use both,
its rank and that of its inverse, for example,
the inverse of a polynomial of rank $n\ge 2$ has rank $n-1$.
A corollary to the \emph{minimal} inverse (Theorem~\ref{thr:ft.mininv})
is used to identify \emph{trivial} units, that is, units from the
(commutative) ground field. We do not even have to exclude the
(commutative) one-variable case.

Factorization (of rational functions) in the latter
(on the level of \emph{realizations})
is well established in control theory
\cite{Bart2008a}
. Factorization in the non-commutative setting
is discussed in 
\cite{Kaliuzhnyi2009a}
\ and
\cite{Helton2018b}
.

\medskip
After fixing the basic notation and stating the basic definitions
in Section~\ref{sec:ft.pl}, we develop the main (technical) tools
in Section~\ref{sec:ft.ro}. In a first reading only
Proposition~\ref{pro:ft.ratop} (rational operations) and
Theorem~\ref{thr:ft.mininv} (minimal inverse) are important.
The main part is Section~\ref{sec:ft.ft}
where the factorization theory is developed,
starting with Definition~\ref{def:ft.factors}
and culminating (but not ending) in Theorem~\ref{thr:ft.ldivs}.
Finally, in Section~\ref{sec:ft.mf}, minimal multiplication
(Theorem~\ref{thr:ft.minmul}) and
factorization (Theorem~\ref{thr:ft.factorization})
is discussed.

\medskip
\begin{remark}
This exposition is not meant to serve as an introduction,
neither to free fields nor to non-commutative factorization (in free
associative algebras). Instead, depending on the background,
the example in \cite[Section~4]{Schrempf2017b9}
, the connection to formal power series
\cite[Section~3]{Schrempf2017a9}
\ or the polynomial factorization
 \cite[Section~2]{Schrempf2017b9}
\ might be helpful. One way to get acquainted with free fields
is to use them (``almost'' like the rational numbers) and explore
the rich theory in parallel. The step from inverting a non-zero
number, say in $s = \frac{v}{a}$ or $as = v$ with \emph{unique} solution $s$,
to inverting ``full'' matrices (Definition~\ref{def:ft.full})
is non-trivial but similar:
$As=v$ with \emph{unique} solution \emph{vector} $s$ (we are usually
interested in its first component $s_1$).
\end{remark}

\section{Preliminaries}\label{sec:ft.pl}

We represent elements (in free fields) by
\emph{admissible linear systems} (Definition~\ref{def:ft.als}),
which are just a special form of \emph{linear representations}
(Definition~\ref{def:ft.rep}) and ``general'' \emph{admissible systems}
\cite[Section~7.1]{Cohn2006a}
. Rational operations
(scalar multiplication, addition, multiplication, inverse) can be
easily formulated in terms of linear representations
(Proposition~\ref{pro:ft.ratop}).

\medskip

\begin{notation}
The set of the natural numbers is denoted by $\numN = \{ 1,2,\ldots \}$,
that including zero by $\numN_0$.
Zero entries in matrices are usually replaced by (lower) dots
to emphasize the structure of the non-zero entries
unless they result from transformations where there
were possibly non-zero entries before.
We denote by $I_n$ the identity matrix
and $\perm_n$ the permutation matrix that reverses the order
of rows/columns (of size $n$)
respectively $I$ and $\perm$ if the size is clear from the context.
The transpose of a vector $v$ is denoted by $v\trp$,
the unit (column) vector with a one at position~$n$ (and size depending on
the context) by $e_n = [\underbrace{0,\ldots,0}_{n-1},1,0\ldots,0]\trp$.
\end{notation}

\medskip
Let $\field{K}$ be a \emph{commutative} field,
$\aclo{\field{K}}$ its algebraic closure and
$X = \{ x_1, x_2, \ldots, x_d\}$ be a \emph{finite} (non-empty) alphabet.
$\freeALG{\field{K}}{X}$ denotes the \emph{free associative
algebra} (or \emph{free $\field{K}$-algebra})
and $\field{F} = \freeFLD{\field{K}}{X}$ its \emph{universal field of
fractions} (or ``free field'') 
\cite{Cohn1995a}
,
\cite{Cohn1999a}
. An element in $\freeALG{\field{K}}{X}$ is called (non-commutative or nc)
\emph{polynomial}.
In our examples the alphabet is usually $X=\{x,y,z\}$.
Including the algebra of \emph{nc rational series}
we have the following chain of inclusions:
\begin{displaymath}
\field{K}\subsetneq \freeALG{\field{K}}{X}
  \subsetneq \ncRATS{\field{K}}{X}
  \subsetneq \freeFLD{\field{K}}{X} =: \field{F}.
\end{displaymath}
The \emph{free monoid} $X^*$ generated by $X$
is the set of all
\emph{finite words}
$x_{i_1} x_{i_2} \cdots x_{i_n}$ with $i_k \in \{ 1,2,\ldots, d \}$.
An element of the alphabet is called \emph{letter},
one of the free monoid \emph{word}.
The multiplication on $X^*$ is the \emph{concatenation}
of words, that is,
$(x_{i_1} \cdots x_{i_m})\cdot (x_{j_1} \cdots x_{j_n})
= x_{i_1} \cdots x_{i_m} x_{j_1} \cdots x_{j_n}$,
with neutral element $1$, the \emph{empty word}.
The \emph{length} of a word $w=x_{i_1} x_{i_2} \cdots x_{i_m}$ is $m$,
denoted by $\length{w} = m$ or $\ell(w) = m$.
For detailed introductions see
\cite[Chapter~1]{Berstel2011a}
\ or
\cite[Section~I.1]{Salomaa1978a}
.

\begin{remark}
For an overview about the many connections to formal language
theory, automata, nc rational series and recognizability
we recommend \cite{Reutenauer2008a}
. \emph{Recognizable} series are defined via ``regular''
\emph{linear representations}
\cite{Berstel2011a}
, a special case of those introduced later
in Definition~\ref{def:ft.rep}.
A summary is available in 
\cite[Section~3]{Schrempf2017a9}
.
\end{remark}

\begin{definition}[Inner Rank, Full Matrix
\protect{\cite[Section~0.1]{Cohn2006a}
}, \cite{Cohn1999a}
]\label{def:ft.full}
Given a matrix $A \in \freeALG{\field{K}}{X}^{n \times n}$, the \emph{inner rank}
of $A$ is the smallest number $m\in \numN$
such that there exists a factorization
$A = T U$ with $T \in \freeALG{\field{K}}{X}^{n \times m}$ and
$U \in \freeALG{\field{K}}{X}^{m \times n}$.
The matrix $A$ is called \emph{full} if $m = n$,
\emph{non-full} otherwise.
\end{definition}

\begin{remark}
Every full matrix (over the free associative algebra) is invertible
over the free field
\cite[Corollary~7.5.14]{Cohn2006a}
.
\end{remark}

\begin{definition}[Linear Representations, Dimension, Rank
\cite{Cohn1994a,Cohn1999a}
]\label{def:ft.rep}
Let $f \in \field{F}$.
A \emph{linear representation} of $f$ is a triple $\pi_f = (u,A,v)$ with
$u \in \field{K}^{1 \times n}$, full
$A = A_0 \otimes 1 + A_1 \otimes x_1 + \ldots
+ A_d \otimes x_d$, that is, $A$ is invertible over $\field{F}$,
$A_\ell \in \field{K}^{n\times n}$,
$v \in \field{K}^{n\times 1}$ 
and $f = u A\inv v$.
The \emph{dimension} of $\pi_f$ is $\dim \, (u,A,v) = n$.
It is called \emph{minimal} if $A$ has the smallest possible dimension
among all linear representations of $f$.
The ``empty'' representation $\pi = (,,)$ is
the minimal one of $0 \in \field{F}$ with $\dim \pi = 0$.
Let $f \in \field{F}$ and $\pi$ be a \emph{minimal}
linear representation of $f$.
Then the \emph{rank} of $f$ is
defined as $\rank f = \dim \pi$.
\end{definition}

\begin{remark}
Cohn and Reutenauer define linear representations
slightly more general, namely $f = c + u A\inv v$ with
possibly non-zero $c \in \field{K}$ and call it
\emph{pure} when $c=0$. Two linear representations are called
\emph{equivalent} if they represent the same element
\cite{Cohn1999a}
. Two (pure) linear representations $(u,A,v)$ and $(\tilde{u},\tilde{A},\tilde{v})$
of dimension $n$ are called \emph{isomorphic} if there exist invertible matrices
$P,Q \in \field{K}^{n \times n}$
such that $u = \tilde{u}Q$, $A = P \tilde{A} Q$ and $v=P\tilde{v}$
\cite{Cohn1999a}
.
\end{remark}

\begin{theorem}[%
\protect{\cite[Theorem 1.4]{Cohn1999a}
}]\label{thr:ft.cohn99.14}
If $\pi' = (u',A',v')$ and $\pi''=(u'',A'',v'')$ are equivalent
(pure) linear representations, of which the first is minimal,
then the second is isomorphic to a representation $\pi = (u,A,v)$
which has the block decomposition
\begin{displaymath}
u =
\begin{bmatrix}
. & u' & * 
\end{bmatrix},\quad
A =
\begin{bmatrix}
* & * & * \\
. & A' & * \\
. & . & * 
\end{bmatrix}
\quad\text{and}\quad
v = 
\begin{bmatrix}
* \\ v' \\ .
\end{bmatrix}.
\end{displaymath}
\end{theorem}

\begin{definition}[Left and Right Families
\cite{Cohn1994a}
]\label{def:ft.family}
Let $\pi=(u,A,v)$ be a linear representation of $f \in \field{F}$
of dimension $n$.
The families $( s_1, s_2, \ldots, s_n )\subseteq \field{F}$
with $s_i = (A\inv v)_i$
and $( t_1, t_2, \ldots, t_n )\subseteq \field{F}$
with $t_j = (u A\inv)_j$
are called \emph{left family} and \emph{right family} respectively.
$L(\pi) = \linsp \{ s_1, s_2, \ldots, s_n \}$ and
$R(\pi) = \linsp \{ t_1, t_2, \ldots, t_n \}$
denote their linear spans (over $\field{K}$).
\end{definition}

\begin{proposition}[%
\protect{\cite[Proposition~4.7]{Cohn1994a}
}]\label{pro:ft.cohn94.47}
A representation $\pi=(u,A,v)$ of an element $f \in \field{F}$
is minimal if and only if both, the left family
and the right family are $\field{K}$-linearly independent.
In this case, $L(\pi)$ and $R(\pi)$ depend only on $f$.
\end{proposition}

\begin{definition}[Element Types]
An element $f \in \field{F}$ is called \emph{of type} $(1,*)$
(respectively $(0,*)$) if $1 \in R(f)$, that is, $1 \in R(\pi)$
for some \emph{minimal} linear representation $\pi$ of $f$,
(respectively $1 \notin R(f)$).
It is called \emph{of type} $(*,1)$ (respectively $(*,0)$)
if $1 \in L(f)$ (respectively $1 \notin L(f)$).
Both subtypes can be combined.
\end{definition}

\begin{definition}[Admissible Linear Systems, Admissible Transformations
\cite{Schrempf2017a9}
]\label{def:ft.als}
A linear representation $\als{A} = (u,A,v)$ of $f \in \field{F}$
is called \emph{admissible linear system} (ALS) for $f$,
written also as $A s = v$,
if $u=e_1=[1,0,\ldots,0]$.
The element $f$ is then the first component
of the (unique) solution vector $s$.
Given a linear representation $\als{A} = (u,A,v)$
of dimension $n$ of $f \in \field{F}$
and invertible matrices $P,Q \in \field{K}^{n\times n}$,
the transformed $P\als{A}Q = (uQ, PAQ, Pv)$ is
again a linear representation (of $f$).
If $\als{A}$ is an ALS,
the transformation $(P,Q)$ is called
\emph{admissible} if the first row of $Q$ is $e_1 = [1,0,\ldots,0]$.
\end{definition}

\begin{remark}
The left family $(A\inv v)_i$ (respectively the right family $(u A\inv)_j$)
and the solution vector $s$ of $As = v$ (respectively $t$ of $u = tA$)
are used synonymously.
\end{remark}

Transformations can be done by elementary row- and column operations,
explained in detail in
\cite[Remark~1.12]{Schrempf2017a9}
. For further remarks and connections to the related concepts
of linearization and realization see
\cite[Section~1]{Schrempf2017a9}
. 

For elements in the free associative algebra $\freeALG{\field{K}}{X}$
a special form (with an upper unitriangular system matrix) can be used.
It plays a crucial role in the factorization of polynomials
because it allows to formulate a minimal polynomial multiplication
(Proposition~\ref{pro:ft.minmul}) and \emph{upper unitriangular}
transformation matrices (invertible by definition) suffice to find
all possible factors (up to trivial units).
For details we refer to 
\cite[Section~2]{Schrempf2017b9}
.

\begin{remark}
The following definition is slightly adapted to avoid confusion with
other transformation matrices for the factorization,
formulated \emph{independent} of a given admissible linear system.
\end{remark}

\begin{definition}[Polynomial ALS and Transformation
\protect{\cite[Definition~24]{Schrempf2017b9}
}]\label{def:ft.psals}
An ALS $\als{A} = (u,A,v)$ of dimension $n$
with system matrix $A = (a_{ij})$
for a non-zero polynomial $0 \neq p \in \freeALG{\field{K}}{X}$ 
is called
\emph{polynomial}, if
\begin{itemize}
\item[(1)] $v = [0,\ldots,0,\lambda]\trp$ for some $\lambda \in\field{K}$ and
\item[(2)] $a_{ii}=1$ for $i=1,2,\ldots, n$ and $a_{ij}=0$ for $i>j$,
  that is, $A$ is upper triangular.
\end{itemize}
A polynomial ALS is also written as $\als{A} = (1,A,\lambda)$
with $1,\lambda \in \field{K}$.
An admissible transformation $(P,Q)$ for an ALS $\als{A}$
is called \emph{polynomial} if it has the form
\begin{displaymath}
(P,Q) = \left(
\begin{bmatrix}
1 & \alpha_{1,2} & \ldots & \alpha_{1,n-1} & \alpha_{1,n} \\
  & \ddots & \ddots & \vdots & \vdots \\
  &   & 1 & \alpha_{n-2,n-1} & \alpha_{n-2,n} \\
  &   &   & 1 & \alpha_{n-1,n} \\
  &   &   &   & 1
\end{bmatrix},
\begin{bmatrix}
1 & 0 & 0 & \ldots & 0 \\
  & 1 & \beta_{2,3} & \ldots & \beta_{2,n} \\
  &   & 1 & \ddots & \vdots  \\
  &   &   & \ddots & \beta_{n-1,n} \\
  &   &   &   & 1 \\
\end{bmatrix}
\right).
\end{displaymath}
If additionally $\alpha_{1,n} = \alpha_{2,n} = \ldots = \alpha_{n-1,n} = 0$
then $(P,Q)$ is called \emph{polynomial factorization transformation}.
\end{definition}

\begin{definition}[Similar Right Ideals, Similar Elements
\protect{\cite[Section~3.1]{Cohn2006a}
}]
Let $R$ be a ring. 
Two right ideals $\ideal{a},\ideal{b} \subseteq R$ are called
\emph{similar}, written as $\ideal{a} \sim \ideal{b}$,
if $R/\ideal{a} \cong R/\ideal{b}$ as right $R$-modules.
Two elements $p,q\in R$ are called \emph{similar} if
their right ideals $pR$ and $qR$ are similar,
that is, $pR \sim qR$.
See also 
\cite[Section~4.1]{Smertnig2016a}
.
\end{definition}

\begin{definition}[Left and Right Coprime Elements
\protect{\cite[Section~2]{Baeth2015a}
}]\label{def:ft.lrcop}
Let $R$ be a domain and $H = R^\bullet = R \setminus \{ 0 \}$.
An element $p$ \emph{left divides} $q$,
written as $p \ldivs q$, if $q \in pH = \{ ph \mid h \in H \}$.
Two elements $p,q$ are called \emph{left coprime}
if for all $h$ such that $h\ldivs p$ and $h\ldivs q$
implies $h \in H^\times = \{ f \in H \mid f \text{ is invertible} \}$,
that is, $h$ is an element of the \emph{group of units}.
Right division $p \rdivs q$ and the notion of
\emph{right coprime} is defined in a similar way.
Two elements are called \emph{coprime}
if they are left and right coprime.
\end{definition}

\begin{definition}[Atomic Domains
\protect{\cite[Section~2]{Baeth2015a}
}]\label{def:ft.atoms}
Let $R$ be a domain and $H = R^\bullet$.
An element $p\in H \setminus H^\times$, that is,
a non-zero non-unit (in $R$), is called an \emph{atom}
(or \emph{irreducible})
if $p = q_1 q_2$ with $q_1,q_2 \in H$ implies that
either $q_1 \in H^\times$ or $q_2 \in H^\times$.
The set of atoms in $R$ is denoted by
$\atoms(R)$.
The (cancellative) monoid $H$ is called \emph{atomic}
if every non-unit can be written as a finite product of atoms of $H$.
The domain $R$ is called \emph{atomic} if the monoid $R^\bullet$ is atomic.
\end{definition}

\begin{definition}[Similarity Unique Factorization Domains
\cite{Smertnig2016a}
]\label{def:ft.sfd}
A domain $R$ is called \emph{similarity factorial}
(or a \emph{similarity-UFD}) if
$R$ is atomic and it satisfies the property that
if $p_1 p_2 \cdots p_m = q_1 q_2 \cdots q_n$ for atoms
(irreducible elements)
$p_i,q_j \in R$,
then $m=n$ and there exists a permutation $\sigma \in \mathfrak{S}_m$
such that $p_i$ is similar to $q_{\sigma(i)}$ for all $i\in 1,2, \ldots, m$.
\end{definition}

\begin{proposition}[%
\protect{\cite[Theorem~6.3]{Cohn1963b}
}]\label{pro:ft.cohn63b}
The free associative algebra $R = \freeALG{\field{K}}{X}$
is a similarity (unique) factorization domain.
\end{proposition}

\section{Rational Operations}\label{sec:ft.ro}

Usually we want to construct \emph{minimal} admissible linear
systems (out of minimal ones), that is,
perform ``minimal'' rational operations.
Minimal scalar multiplication is trivial.
In some special cases \emph{minimal addition}
can be formulated (Proposition~\ref{pro:ft.djminadd}).
For \emph{minimal multiplication} we refer to Section~\ref{sec:ft.mf}.
For the \emph{minimal inverse} we have to distinguish four
cases, which are summarized in Theorem~\ref{thr:ft.mininv}.
In general however, it is necessary to \emph{minimize} a given
system. For a \emph{polynomial} ALS this is discussed in
\cite[Section~2.2]{Schrempf2017b9}
, for the general case we refer to
\cite{Schrempf2018a}
.

\begin{proposition}[Rational Operations
\cite{Cohn1999a}
]\label{pro:ft.ratop}
Let $0\neq f,g \in \field{F}$ be given by the
admissible linear systems $\als{A}_f = (u_f, A_f, v_f)$
and $\als{A}_g = (u_g, A_g, v_g)$ respectively
and let $0\neq \mu \in \field{K}$.
Then admissible linear systems for the rational operations
can be obtained as follows:

\smallskip\noindent
The scalar multiplication
$\mu f$ is given by
\begin{displaymath}
\mu \als{A}_f =
\bigl( u_f, A_f, \mu v_f \bigr).
\end{displaymath}
The sum $f + g$ is given by
\begin{displaymath}
\als{A}_f + \als{A}_g =
\left(
\begin{bmatrix}
u_f & . 
\end{bmatrix},
\begin{bmatrix}
A_f & -A_f u_f\trp u_g \\
. & A_g
\end{bmatrix}, 
\begin{bmatrix} v_f \\ v_g \end{bmatrix}
\right).
\end{displaymath}
The product $fg$ is given by
\begin{displaymath}
\als{A}_f \cdot \als{A}_g =
\left(
\begin{bmatrix}
u_f & . 
\end{bmatrix},
\begin{bmatrix}
A_f & -v_f u_g \\
. & A_g
\end{bmatrix},
\begin{bmatrix}
. \\ v_g
\end{bmatrix}
\right).
\end{displaymath}
And the inverse $f\inv$ is given by
\begin{displaymath}
\als{A}_f\inv =
\left(
\begin{bmatrix}
1 & . 
\end{bmatrix},
\begin{bmatrix}
-v_f & A_f \\
. & u_f
\end{bmatrix},
\begin{bmatrix}
. \\ 1
\end{bmatrix}
\right).
\end{displaymath}
\end{proposition}

\begin{definition}[Disjoint Elements
\cite{Cohn1999a}
]\label{def:ft.disjoint}
Two elements $f,g \in \field{F}$ are called
\emph{disjoint} if $\rank(f+g) = \rank(f) + \rank(g)$.
\end{definition}

\begin{remark}
Two polynomials are never disjoint.
This can be easily seen in the construction
(of an ALS) for the sum (of the polynomials).
See also \cite[Theorem~2.3]{Cohn1999a}
.
\end{remark}

For disjoint elements the formulation of a minimal addition
(Proposition~\ref{pro:ft.djminadd}) is immediate.
Testing if two elements are disjoint in $\field{F}$ is difficult
because it relies on techniques for minimizing linear representations
\cite{Schrempf2018a}
.
However, since minimality of a linear representation
is equivalent to $\field{K}$-linear independence of
its left and right family respectively (Proposition~\ref{pro:ft.cohn94.47}),
two elements $f,g$ are disjoint if
$L(f) \cap L(g) = \{ 0 \}$ and $R(f) \cap R(g) = \{ 0 \}$.

\begin{example}
For $f = x + \bigl( (1-x)\inv + x\inv \bigr)$
a \emph{minimal} ALS
---constructed by Proposition~\ref{pro:ft.djminadd}--- is
\begin{displaymath}
\begin{bmatrix}
1 & -x & -1 & . \\
. & 1 & . & . \\
. & . & 1-x & x+1 \\
. & . & . & x
\end{bmatrix}
s =
\begin{bmatrix}
. \\ 1 \\ 1 \\ 1
\end{bmatrix},\quad
s =
\begin{bmatrix}
f \\
1 \\
(1+x)\inv + x\inv \\
x\inv
\end{bmatrix}.
\end{displaymath}
\end{example}

\begin{proposition}[Minimal Disjoint Addition]\label{pro:ft.djminadd}
Let $f,g \in \field{F}$ be \emph{disjoint} and
given by the \emph{minimal}
admissible linear systems $\als{A}_f = (u_f, A_f, v_f)$
and $\als{A}_g = (u_g, A_g, v_g)$ of dimension $n_f$ and $n_g$
respectively.
Then the system
\begin{displaymath}
\als{A}_f + \als{A}_g =
\left(
\begin{bmatrix}
u_f & . 
\end{bmatrix},
\begin{bmatrix}
A_f & -A_f u_f\trp u_g \\
. & A_g
\end{bmatrix}, 
\begin{bmatrix} v_f \\ v_g \end{bmatrix}
\right)
\end{displaymath}
of dimension $n_f+n_g$
(from Proposition~\ref{pro:ft.ratop})
for $f+g$ is minimal.
\end{proposition}

Like in the polynomial case, factorization and \emph{minimal}
multiplication are tight together as opposite points of view.
Further assumptions that guarantee minimality are developed
in Section~\ref{sec:ft.ft}.
They eventually enter in Theorem~\ref{thr:ft.minmul}.
We will need alternative constructions
(to that in Proposition~\ref{pro:ft.ratop})
for the product
several times, so we state them already here
in Propositions~\ref{pro:ft.mul2} and~\ref{pro:ft.mul1}.
These constructions are used in particular in Theorem~\ref{thr:ft.minmul}.
Before, we need some technical results from
\cite{Schrempf2017a9}
\ and
\cite{Schrempf2017b9}
. However these are rearranged such that similarities become
more obvious and the flexibility in applications is increased.
In particular we prove Lemma~\ref{lem:ft.for1} by applying
Lemma~\ref{lem:ft.rt1}.

\begin{lemma}[\protect{%
\cite[Lemma~25]{Schrempf2017b9}
}]\label{lem:ft.rt1}
Let $\als{A} = (u,A,v)$ be an ALS
of dimension $n \ge 1$ with $\field{K}$-linearly independent
left family $s = A\inv v$ and
$B = B_0 \otimes 1 + B_1 \otimes x_1 + \ldots + B_d \otimes x_d$
with $B_\ell\in \field{K}^{m\times n}$, such that
$B s = 0$. Then there exists a (unique) $T \in \field{K}^{m \times n}$
such that $B = TA$.
\end{lemma}

\begin{lemma}\label{lem:ft.rt2}
Let $\als{A} = (u,A,v)$ be an ALS 
of dimension $n \ge 1$ with $\field{K}$-linearly independent
right family $t = u A\inv$ and
$B = B_0 \otimes 1 + B_1 \otimes x_1 + \ldots + B_d \otimes x_d$
with $B_\ell\in \field{K}^{n\times m}$, such that
$t B = 0$. Then there exists a (unique) $U \in \field{K}^{n \times m}$
such that $B = AU$.
\end{lemma}

\begin{lemma}[for Type~$(0,1)$
\protect{\cite[Lemma~4.11]{Schrempf2017a9}
}]\label{lem:ft.for1}
Let $\als{A} = (u,A,v)$ be a \emph{minimal} ALS
with $\dim \als{A} =n \ge 2$ and $1\in L(\als{A})$. Then there
exists an admissible transformation $(P,Q)$ such that
the last row of $PAQ$ is $[0,\ldots,0,1]$
and $Pv = [0,\ldots,0,\lambda]\trp$
for some $\lambda \in \field{K}$.
\end{lemma}

\begin{proof}
Without loss of generality, assume that $v = [0,\ldots,0,1]\trp$ and
the left family $s = A\inv v$ is $(s_1, s_2,$ $\ldots, s_{n-1}, 1)$.
Otherwise it can be brought to this form by some admissible transformation
$(P^\circ, Q^\circ)$.
Now let $\bar{A}$ denote the upper left $(n-1)\times (n-1)$ block of $A$,
let $\bar{s} = (s_1, \ldots, s_{n-1})$ and write $As = v$ as
\begin{displaymath}
\begin{bmatrix}
\bar{A} & b \\
c & d
\end{bmatrix}
\begin{bmatrix}
\bar{s} \\ 1
\end{bmatrix}
=
\begin{bmatrix}
0 \\ 1
\end{bmatrix}.
\end{displaymath}
Now let $B = [ c, d-1]$ and apply Lemma~\ref{lem:ft.rt1}
to get the matrix $T = [\bar{T}, \tau ] \in \field{K}^{1 \times n}$
such that $B = TA$. Thus we get the transformation
\begin{displaymath}
(P,Q) = \left(
\begin{bmatrix}
I_{n-1} & . \\
-\bar{T} & -\tau
\end{bmatrix}
P^\circ,
Q^\circ
\right).
\end{displaymath}
\end{proof}

\begin{lemma}[for Type~$(1,0)$
\protect{\cite[Lemma~4.12]{Schrempf2017a9}
}]\label{lem:ft.for2}
Let $\als{A} = (u,A,v)$ be a \emph{minimal} ALS
with $\dim \als{A} = n \ge 2$ and $1\in R(\als{A})$. Then there
exists an admissible transformation $(P,Q)$ such that
the first column of $PAQ$ is $[1,0,\ldots,0]\trp$
and $Pv = [0,\ldots,0,\lambda]\trp$
for some $\lambda \in \field{K}$.
\end{lemma}

\begin{remark}
If $g$ is of type $(*,1)$ then, by Lemma~\ref{lem:ft.for1},
\emph{each} minimal ALS for $g$ can be transformed into one
with a last row of the form $[0,\ldots,0,1]$.
If $g$ is of type $(1,*)$ then, by Lemma~\ref{lem:ft.for2},
\emph{each} minimal ALS for $g$ can be transformed into one
with a first column of the form $[1,0,\ldots,0]\trp$.
This can be done by \emph{linear} techniques,
see the remark before
\cite[Theorem~4.13]{Schrempf2017a9}
.
\end{remark}

\medskip
Since $p \in \freeALG{\field{K}}{X}$ is of type $(1,1)$,
both constructions
can be used for the minimal polynomial multiplication
(Proposition~\ref{pro:ft.minmul}).
One could call the multiplication from Proposition~\ref{pro:ft.ratop}
type $(*,*)$. A \emph{necessary} condition for minimality however
is, that the left factor is of type~$(*,0)$ and
the right factor is of type~$(0,*)$,
thus we will use this construction later as type~$(0,0)$.
Section~\ref{sec:ft.ft} is dedicated to a \emph{sufficient} condition.
See also Figure~\ref{fig:ft.minmul}, page~\pageref{fig:ft.minmul}.

\begin{proposition}[Multiplication Type $(1,*)$]\label{pro:ft.mul2}
Let $f,g\in \field{F} \setminus \field{K}$ be given by the
admissible linear systems
$\als{A}_f = (u_f, A_f, v_f) = (1,A_f,\lambda_f)$ of dimension $n_f$
of the form
\begin{equation}\label{eqn:ft.mul2.f}
\als{A}_f = \left(
\begin{bmatrix}
1 & . & . 
\end{bmatrix},
\begin{bmatrix}
a & b' & b \\
a' & B & b'' \\
. & . & 1
\end{bmatrix},
\begin{bmatrix}
. \\ . \\ \lambda_f
\end{bmatrix}
\right)
\end{equation}
and $\als{A}_g = (u_g, A_g, v_g) = (1,A_g,\lambda_g)$
of dimension $n_g$ respectively.
Then an ALS for $fg$ of dimension $n = n_f + n_g - 1$ is given by
\begin{equation}\label{eqn:ft.mul2.fg}
\als{A} = \left(
\begin{bmatrix}
1 & . & . 
\end{bmatrix},
\begin{bmatrix}
a & b' & \lambda_f b u_g \\
a' & B & \lambda_f b'' u_g \\
. & . & A_g
\end{bmatrix},
\begin{bmatrix}
. \\ . \\ v_g
\end{bmatrix}
\right).
\end{equation}
\end{proposition}

\begin{proof}
Construct the ALS $\als{A}'=(u',A',v')$ of dimension $n_f+n_g$
for the product $fg$
using Proposition~\ref{pro:ft.ratop}. Add $\lambda_f$-times
column~$n_f$ to column~$(n_f+1)$ (in the system matrix $A'$).
Remove column~$n_f$ from $A'$ and $v'$ and row~$n_f$
from $A'$ and $u'$ to get the ALS~\eqref{eqn:ft.mul2.fg}
of dimension $n_f+n_g-1$.
\end{proof}

\begin{proposition}[Multiplication Type $(*,1)$]\label{pro:ft.mul1}
Let $f,g\in \field{F} \setminus \field{K}$ be given by the
admissible linear systems $\als{A}_f = (u_f, A_f, v_f) = (1, A_f, \lambda_f)$ 
of dimension $n_f$ and
$\als{A}_g = (u_g, A_g, v_g) = (1, A_g, \lambda_g)$
of dimension $n_g$ of the form
\begin{equation}\label{eqn:ft.mul1.g}
\als{A}_g = \left(
\begin{bmatrix}
1 & . & .
\end{bmatrix},
\begin{bmatrix}
1 & b'& b \\
. & B & b'' \\
. & c' & c
\end{bmatrix},
\begin{bmatrix}
. \\ . \\ \lambda_g
\end{bmatrix}
\right)
\end{equation}
respectively.
Then an ALS for $fg$ of dimension $n = n_f + n_g -1$ is given by
\begin{equation}\label{eqn:ft.mul1.fg}
\als{A} = \left(
\begin{bmatrix}
u_f & . & . 
\end{bmatrix},
\begin{bmatrix}
A_f & e_{n_f} \lambda_f b' & e_{n_f} \lambda_f b \\
. & B & b'' \\
. & c' & c
\end{bmatrix},
\begin{bmatrix}
. \\ . \\ \lambda_g
\end{bmatrix}
\right).
\end{equation}
\end{proposition}

\begin{proof}
Construct the ALS $\als{A}'=(u',A',v')$ of dimension $n_f+n_g$
for the product $fg$
using Proposition~\ref{pro:ft.ratop}. Add $\lambda_f$-times
row~$(n_f+1)$ to row~$n_f$ (in the system matrix $A'$).
Remove row~$(n_f+1)$ from $A'$ and $v'$ and column~$(n_f+1)$
from $A'$ and $u'$ to get the ALS~\eqref{eqn:ft.mul1.fg}
of dimension $n_f+n_g-1$.
\end{proof}

\begin{remark}
Recall that, if $f$ (respectively $g$) is given by a \emph{minimal} ALS,
it can be transformed appropriately by Lemma~\ref{lem:ft.for1}
(respectively Lemma~\ref{lem:ft.for2})
into the form \eqref{eqn:ft.mul2.f} (respectively \eqref{eqn:ft.mul1.g}).
\end{remark}

Lemma~\ref{lem:ft.min1} is a slightly more general version of
\cite[Lemma~27]{Schrempf2017b9}
. 
The proof of (the following) Proposition~\ref{pro:ft.minmul}
becomes simple
by the help of the two lemmas~\ref{lem:ft.slinin}
and~\ref{lem:ft.tlinin} which are extracted of the original
proof (of the minimal polynomial multiplication).
They are useful later,
especially in Lemma~\ref{lem:ft.special}.

\begin{remark}
Note that the transformation in the following lemma is \emph{not}
necessarily admissible. However, except for $n=2$ (which can be
treated by permuting the last two elements in the left family),
it can be chosen such that it is \emph{admissible}.
\end{remark}

\begin{lemma}\label{lem:ft.min1}
Let $\als{A} = (u,A,v)$ be an ALS of dimension $n\ge 2$
with $v = [0,\ldots,0,\lambda]\trp$ and
$\field{K}$-linearly dependent left family $s=A\inv v$.
Let $m \in \{ 2, 3, \ldots, n \}$ be the minimal index
such that the left subfamily $\underline{s} = (A\inv v)_{i=m}^n$
is $\field{K}$-linearly independent.
Let $A = (a_{ij})$ and assume that $a_{ii}=1$ for $1 \le i \le m$
and $a_{ij}=0$ for $j < i \le m$ (upper triangular $m \times m$ block)
and $a_{ij}=0$ for $j \le m < i$ (lower left zero block of size $(n-m) \times m$).
Then there exists matrices $T,U \in \field{K}^{1 \times (n+1-m)}$
such that
\begin{displaymath}
U + (a_{m-1,j})_{j=m}^n - T(a_{ij})_{i,j=m}^n  =
\begin{bmatrix}
  0 & \ldots & 0
\end{bmatrix}
\quad\text{and}\quad
T(v_i)_{i=m}^n = 0.
\end{displaymath}
\end{lemma}

\begin{proof}
By assumption, the left subfamily $(s_{m-1}, s_m, \ldots, s_n)$
is $\field{K}$-linearly dependent.
Thus there are $\kappa_m, \ldots, \kappa_n \in \field{K}$ such that
$s_{m-1} = \kappa_m s_m + \kappa_{m+1} s_{m+1} + \ldots + \kappa_n s_n$.
Let $U = [\kappa_m, \kappa_{m+1}, \ldots, \kappa_n ]$.
Then $s_{m-1} - U \underline{s} = 0$.
By assumption $v_{m-1}=0$.
Now we can apply Lemma~\ref{lem:ft.rt1} with
$B = U + [ a_{m-1,m}, a_{m-1,m+1}, \ldots, a_{m-1,n}]$
(and $\underline{s}$).
Hence, there exists a matrix
$T\in \field{K}^{1 \times (n+1-m)}$
such that
\begin{displaymath}
U + 
\begin{bmatrix}
a_{m-1,m} & \ldots & a_{m-1,n}
\end{bmatrix}
- T
\begin{bmatrix}
a_{m,m} & \ldots & a_{m,n} \\
\vdots & \ddots & \vdots \\
a_{n,m} & \ldots & a_{n,n}
\end{bmatrix}
=
\begin{bmatrix}
0 & \ldots & 0 
\end{bmatrix}
\end{displaymath}
holds. Recall that the last column of $T$ is zero,
whence $T(v_i)_{i=m}^n = 0$.
\end{proof}

\begin{lemma}\label{lem:ft.slinin}
Let $p \in \freeALG{\field{K}}{X} \setminus \field{K}$ and
$g \in \field{F} \setminus \field{K}$
be given by the \emph{minimal} admissible linear systems
$A_p = (u_p, A_p, v_p)$ and $A_g = (u_g, A_g, v_g)$
of dimension $n_p$ and $n_g$ respectively
with $1 \in R(g)$.
Then the left family of the ALS $\als{A} = (u,A,v)$
for $pg$ of dimension $n = n_p + n_g -1$ from
Proposition~\ref{pro:ft.mul1} is $\field{K}$-linearly independent.
\end{lemma}

\begin{proof}
Without loss of generality assume $v = [0,\ldots,0,1]\trp$,
$\als{A}_p$ in polynomial form
(by \cite[Proposition~2.1~(ii)]{Cohn1999a}
\ every polynomial admits a linear representation
with upper unitriangular system matrix)
and $A_g$ with first
column $[1,0,\ldots,0]\trp$.
Let $s_p = (s^p_1, \ldots, s^p_{n_p})$ and
$s_g = (s^g_1, \ldots, s^g_{n_g})$ be the respective
left family of $\als{A}_p$ and $\als{A}_g$.
We have to show that the left family
\begin{displaymath}
s = (s_1, s_2, \ldots, s_n)
  = (s^p_1 g, \ldots, s^p_{n_p-1}g, g, s^g_2, \ldots, s^g_{n_g}).
\end{displaymath}
of $\als{A}$ is $\field{K}$-linearly independent.
Assume to the contrary that there is an index $1 < m \le n_p$
such that $(s_{m-1}, s_m, \ldots, s_n)$ is $\field{K}$-linearly dependent
while $(s_m,\ldots,s_n)$ is $\field{K}$-linearly independent.
Then, by Lemma~\ref{lem:ft.min1},
there exist matrices $T,U \in \field{K}^{1 \times (n-m+1)}$
as blocks in (invertible) matrices $P,Q \in \field{K}^{n \times n}$,
\begin{displaymath}
P = 
\begin{bmatrix}
I_{m-2} & . & . \\
. & 1 & T \\
. & . & I_{n-m+1}
\end{bmatrix}
\quad\text{and}\quad
Q =
\begin{bmatrix}
I_{m-2} & . & . \\
. & 1 & U \\
. & . & I_{n-m+1}
\end{bmatrix},
\end{displaymath}
that yield equation $s_{m-1}=0$ (in row~$m-1$) in $P\als{A}Q$.
(This ``potential'' transformation $(P,Q)$ is not
necessarily admissible. But this is not an issue here,
since we are only checking linear independence
of the left family.)
Let $\tilde{P}$ (respectively $\tilde{Q}$) be the upper left part
of $P$ (respectively $Q$) of size $n_g \times n_g$.
Then the equation in row~$m-1$ in $\tilde{P} \als{A}_p \tilde{Q}$
is $s^p_{m-1} = \alpha \in \field{K}$, contradicting $\field{K}$-linear
independence of the left family of $\als{A}_p$ since
$s^p_{n_p} = \lambda \in \field{K}$.
\end{proof}

\begin{remark}
Nothing can be said about minimality of $\als{A}$ since
the right family $t = u A\inv$ could be $\field{K}$-linearly
dependent. As an example take $p=xy$ and $g = y\inv + z$.
An ALS for $pg = x + xyz$ constructed by Proposition~\ref{pro:ft.mul1} is
\begin{displaymath}
\begin{bmatrix}
1 & - x & . & . & . \\
. & 1 & -y & . & . \\
. & . & 1 & 1 & -z \\
. & . & . & y & 1 \\
. &. & . & . & 1
\end{bmatrix}
s =
\begin{bmatrix}
. \\ . \\ . \\ . \\ 1
\end{bmatrix}.
\end{displaymath}
The right family is $t = [1, x, xy, x, x+xyz]$.
\end{remark}

\begin{lemma}\label{lem:ft.tlinin}
Let $f \in \field{F} \setminus \field{K}$
and $q \in \freeALG{\field{K}}{X} \setminus \field{K}$
be given by the \emph{minimal} admissible linear systems
$A_f = (u_f, A_f, v_f)$ and $A_q = (u_q, A_q, v_q)$
of dimension $n_f$ and $n_q$ respectively
with $1 \in L(f)$.
Then the right family of the ALS $\als{A} = (u,A,v)$
for $fq$ of dimension $n = n_f + n_q -1$ from
Proposition~\ref{pro:ft.mul2} is $\field{K}$-linearly independent.
\end{lemma}

\begin{proposition}[Minimal Polynomial Multiplication
\protect{\cite[Proposition~28]{Schrempf2017b9}
}]\label{pro:ft.minmul}
Let $p,q \in \freeALG{\field{K}}{X}$ be given by the
\emph{minimal} polynomial admissible linear systems
$A_p = (1, A_p, \lambda_p)$ and
$A_q = (1, A_q, \lambda_q)$ of dimension $n_p,n_q \ge 2$ respectively.
Then the ALS $\als{A}$ from Proposition~\ref{pro:ft.mul2} for $pq$
is \emph{minimal} of dimension $n = n_p + n_q - 1$.
\end{proposition}

\begin{proof}
The left family of $\als{A}$ is $\field{K}$-linearly independent by
Lemma~\ref{lem:ft.slinin} and its right family is $\field{K}$-linearly
independent by Lemma~\ref{lem:ft.tlinin}.
Whence $\als{A}$ is minimal (by Proposition~\ref{pro:ft.cohn94.47})
and by construction in polynomial form.
\end{proof}

\begin{theorem}[Minimal Inverse
\protect{\cite[Theorem~4.13]{Schrempf2017a9}
}]\label{thr:ft.mininv}
Let $f \in \field{F} \setminus \field{K}$ be given by
the \emph{minimal} admissible linear system $\als{A} = (u, A, v)$
of dimension $n$.
Then a \emph{minimal} ALS for $f\inv$ is given
in the following way:

\smallskip\noindent
$f$ of type $(1,1)$ yields $f\inv$ of type $(0,0)$ with $\dim(\als{A}') = n-1$:
\begin{equation}\label{eqn:ft.inv3b}
\als{A}'=\left(1, 
\begin{bmatrix}
-\lambda \perm b'' & -\perm B \perm \\
-\lambda b & - b'\perm
\end{bmatrix},
1 \right) \quad\text{for}\quad
\als{A} = \left(1, 
\begin{bmatrix}
1 & b' & b \\
. & B & b'' \\
. & . & 1
\end{bmatrix},
\lambda \right).
\end{equation}
$f$ of type $(1,0)$ yields $f\inv$ of type $(1,0)$ with $\dim(\als{A}') = n$:
\begin{equation}\label{eqn:ft.inv2b}
\als{A}' = \left(1,
\begin{bmatrix}
1 & - \frac{1}{\lambda} c & - \frac{1}{\lambda} c' \perm \\
. & -\perm b'' & -\perm B \perm \\
. & -b & -b'\perm 
\end{bmatrix},
1 \right)
\quad\text{for}\quad
\als{A} = \left(1,
\begin{bmatrix}
1 & b' & b \\
. & B & b'' \\
. & c' & c
\end{bmatrix},
\lambda \right).
\end{equation}
$f$ of type $(0,1)$ yields $f\inv$ of type $(0,1)$ with $\dim(\als{A}') = n$:
\begin{equation}\label{eqn:ft.inv1b}
\als{A}' = \left( 1,
\begin{bmatrix}
-\lambda \perm b'' & -\perm B \perm & -\perm a' \\
-\lambda b & -b'\perm & -a \\
. & . & 1 
\end{bmatrix},
1 \right)
\quad\text{for}\quad
\als{A} = \left(1,
\begin{bmatrix}
a & b' & b \\
a' & B & b'' \\
. & . & 1
\end{bmatrix},
\lambda \right).
\end{equation}
$f$ of type $(0,0)$ yields $f\inv$ of type $(1,1)$ with $\dim(\als{A}') = n+1$:
\begin{equation}\label{eqn:ft.inv0}
\als{A}' = \left(1,
\begin{bmatrix}
\perm v & -\perm A \perm \\
. & u \perm
\end{bmatrix},
1 \right).
\end{equation}
(Recall that the permutation matrix $\perm$ reverses the order of rows/columns.)
\end{theorem}

\begin{corollary}\label{cor:ft.rank1}
Let $0 \neq f \in \field{F}$.
Then $f \in \field{K}$ if and only if $\rank(f) = \rank(f\inv) = 1$.
\end{corollary}

\begin{remark}
This simple consequence of Theorem~\ref{thr:ft.mininv} makes it possible
to distinguish between \emph{trivial} units (non-zero scalar elements)
and \emph{non-trivial} units, that is, elements in $\field{F} \setminus \field{K}$.
The main idea in the factorization theory
in Section~\ref{sec:ft.ft} is to allow only (the insertion of) \emph{trivial}
units (in factorizations). It is used explicitly in Lemma~\ref{lem:ft.special}
and implicitly in Theorem~\ref{thr:ft.ldivs}.
\end{remark}

\begin{remark}
Note that $n\ge 2$ for type $(1,1)$, $(1,0)$ and $(0,1)$.
The block $B$ is always square of size $n-2$.
For $n=2$ the system matrix of $\als{A}$ is
\begin{itemize}
\item
  $\bigl[\begin{smallmatrix} 1 & b \\ . & 1 \end{smallmatrix}\bigr]$
  for type $(1,1)$,
\item
  $\bigl[\begin{smallmatrix} 1 & b \\ . & c \end{smallmatrix}\bigr]$
  for type $(1,0)$ and
\item
  $\bigl[\begin{smallmatrix} a & b \\ . & 1 \end{smallmatrix}\bigr]$
  for type $(0,1)$.
\end{itemize}
\end{remark}

\section{Factorization Theory}\label{sec:ft.ft}

To compensate the lack of non-zero non-units in 
$\field{F} = \freeFLD{\field{K}}{X}$,
that is, $\field{F}\setminus \{ 0\} = \field{F}^\bullet = \field{F}^\times
  = \{ f \in \field{F} \mid f \text{ is invertible} \}$,
we will view the elements in terms of their \emph{minimal}
linear representations. Recall that the dimension of a minimal one
of $f \in \field{F}$ defines the rank of $f$.

Firstly, in Definition~\ref{def:ft.factors},
we define \emph{factors} based on the rank.
Although this definition would suffice to define divisibility
for polynomials, it is too rigid in general.
Since this is far from obvious it is explained in detail
in an example before Definition~\ref{def:ft.divisors}
(left and right divisibility).
Secondly, some preparation is necessary to be able to exclude the
insertion of non-trivial units. This is the essence of Lemma~\ref{lem:ft.special}.
Finally, Theorem~\ref{thr:ft.ldivs} yields, as the main result,
the equivalence of the ``classical'' divisibility (in free associative algebras)
and the new one (for the free field) for polynomials.

\medskip
For a factorization of a (non-zero) polynomial $p = q_1 q_2 \cdots q_m$
into atoms $q_i$ we would like to have a factorization of its
inverse $p\inv = (q_1 q_2 \cdots q_m)\inv = q_m\inv \cdots q_2\inv q_1\inv$
into atoms $q_i\inv$.
For two polynomials $p,q$ we have
---due to the minimal polynomial multiplication---
$\rank(p)+\rank(q) = \rank(pq)+1$.
Recalling Definition~\ref{def:ft.lrcop} we
have $p \ldivs h$ if $h = pq$
for some $q \in \freeALG{\field{K}}{X}$.
The minimal inverse type $(1,1)$ yields
\begin{displaymath}
\rank(q\inv) + \rank(p\inv) = \rank(q)-1 + \rank(p)-1 = \rank(q\inv p\inv),
\end{displaymath}
or $q\inv$ ``left divides'' $h\inv$
for $h = pq$.
See Proposition~\ref{pro:ft.minmul}, Theorem~\ref{thr:ft.mininv}
and Lemma~\ref{lem:ft.rank}.
To avoid inserting non-trivial units from $\field{F}\setminus\field{K}$,
we have to bound the sum of the ranks of the two factors:
$\rank(p x) + \rank(x\inv q) = \rank(pq) + 2$.

\begin{definition}[Left and Right Factors]\label{def:ft.factors}
Let $h \in \Fbullet = \field{F}^\bullet$ be given.
An element $f\in \Fbullet$ is called
\emph{left factor} of $h$ if
\begin{align*}
\rank(f) + \rank (f\inv h) &\le \rank(h) + 1
  \quad\text{and} \\
\rank(h\inv f) + \rank(f\inv) &\le \rank(h\inv) + 1.
\end{align*}
An element $g\in \Fbullet$ is called
\emph{right factor} of $h$ if
\begin{align*}
\rank (h g\inv) + \rank(g) &\le \rank(h) + 1
  \quad\text{and} \\
\rank (g\inv) + \rank(g h\inv) &\le \rank(h\inv) + 1.
\end{align*}
Scalars and scalar multiples of $h$ are called \emph{trivial} factors.
Non-trivial left/right factors are called \emph{proper}.
Left and right factors are also called \emph{outer}
to distinguish them from (general) factors of a factorization.
\end{definition}

\begin{remark}
Straight away we have that $f$ is a \emph{left factor}
of $h$ if and only if $g =f\inv h$ is a \emph{right factor} of $h$.
And $f$ is a \emph{left factor} of $h$ if and only if
$f\inv$ is a \emph{right factor} of $h\inv$.
\end{remark}

For two polynomials $p$ and $q$ the previous definition
tells us that $p$ (respectively $q$) is a left (respectively right)
factor of $pq$.
However, in general $f$ is \emph{not}
a left factor of $h = f g$. As an example take $f = (xyz)\inv$ and
$g= x$. Then
\begin{align*}
\rank(f) + \rank(g)
  &= \rank(z\inv y\inv x\inv) + \rank(x) \\
  &= 3 + 2 \\
  &> \rank(z\inv y\inv) + 1.
\end{align*}
While here it is easy to see that $f\inv$ and $g$ have a
non-trivial left divisor in $\freeALG{\field{K}}{X}$
(in the sense of Definition~\ref{def:ft.lrcop}),
this can be much more delicate in general,
illustrated in Example~\ref{ex:ft.factors}.
This example will also show that the definition
of outer factors is rather restrictive
and not applicable directly. Later \emph{left} and
\emph{right divisors} will be defined more generally
in such a way that outer factors can be
``split off'' in at least one possible sequence
(see Definition~\ref{def:ft.divisors}).
Although we will see later that this is a generalization of the factorization
in the free associative algebra, it is much more difficult
to apply for two reasons: One has to test \emph{all}
possible ``sequences'' of factorizations to get the atoms
(up to ``similarity'').
And the invertibility
of the transformation matrices ---to admissibly transform the
ALS in such a way that the factors can be ``extracted''---
has to be ensured by including
a condition for non-vanishing determinant. The latter might
restrict practical applications to rank $\le 6$, similar to the
test if a matrix is full 
\cite{Janko2018a}
. Section~\ref{sec:ft.mf} provides further details.
Experiments show that testing (ir)reducibility of
polynomials (using polynomial admissible linear systems)
works practically for rank $\le 12$,
in some cases up to rank $\le 17$
\cite{Janko2018a}
.

\ifJOURNAL
\else
\newpage
\fi
\begin{Example}\label{ex:ft.factors}
Let $f = f_1 f_2 f_3$, with $f_1 = (xy)\inv$, $f_2 = 1-xz$ and $f_3 = (yz)\inv$,
be given by the \emph{minimal} ALS
\begin{equation}\label{eqn:ft.factors1}
\begin{bmatrix}
y & -1 & z & 0 \\
. & x & -1 & 0 \\
0 & 0 & z & -1 \\
0 & 0 & . & y
\end{bmatrix}
s =
\begin{bmatrix}
0 \\ 0 \\ 0 \\ 1
\end{bmatrix}.
\end{equation}
Then it is immediate (after recalling the construction
of an ALS for the product from Proposition~\ref{pro:ft.ratop})
that $f_3$ is a right factor of $f$
by duplicating $s_3$, that is, inserting a ``dummy'' row
(between row~2 and~3):
\begin{displaymath}
\begin{bmatrix}
y & -1 & z & 0 & . \\
. & x & -1 & 0 & . \\
0 & 0 & 1 & -1 & 0 \\
. & . & 0 & z & -1 \\
. & . & 0 & . & y
\end{bmatrix}
s' =
\begin{bmatrix}
. \\ . \\ 0 \\ . \\ 1
\end{bmatrix},
\quad
s' =
\begin{bmatrix}
s_1 \\ s_2 \\ s_3 \\ s_3 \\ s_4
\end{bmatrix}.
\end{displaymath}
Thus, if a minimal ALS $\als{A} = (u,A,v)$ for $f$
is not of the form
\eqref{eqn:ft.factors1}, we need to find an
admissible transformation $(P,Q)$ such that
the (transformed) system matrix $PAQ$ has
a lower left zero block of size $2 \times 2$
and an upper right zero block of size $2 \times 1$
and only the last component of the right hand side
$Pv$ is non-zero,
to detect the right factor $f_3$.
Similarly, subtracting row~3 from~1 and adding column~2 to~4
in \eqref{eqn:ft.factors1} yields
\begin{displaymath}
\begin{bmatrix}
y & -1 & 0 & 0 \\
. & x & -1 & x \\
0 & 0 & z & -1 \\
0 & 0 & . & y
\end{bmatrix}
s =
\begin{bmatrix}
0 \\ 0 \\ . \\ 1
\end{bmatrix}.
\end{displaymath}
Compare with Figure~\ref{fig:ft.minmul}, page~\pageref{fig:ft.minmul},
$k=2$ in type $(*,1)$.
By duplicating $t_2$, that is, inserting a ``dummy'' column
(between column~2 and~3) one can see that $f_1=(xy)\inv$ is
a left factor of $f = (xy)\inv (1-xz) (yz)\inv$:
\begin{displaymath}
\begin{bmatrix}
1 & . & 0 & . & .
\end{bmatrix}
=
\begin{bmatrix}
t_1 & t_2 & t_2 & t_3 & t_4
\end{bmatrix}
\begin{bmatrix}
y & -1 & 0 & . & . \\
. & x  & -1 & 0 & 0 \\
0 & 0 & 1 & -1 & x \\
. & . & 0 & z & -1 \\
. & . & 0 & . & y
\end{bmatrix}.
\end{displaymath}
However, $f_1$ (respectively $f_2$) is \emph{not} a left
(respectively right) factor of $f_1 f_2$
while $y\inv$ is a left factor of $f_1 f_2$ and $x\inv$ is a
left factor of $x\inv f_2$. We now take a closer look on that
phenomenon. A \emph{minimal} ALS for $f' = f_1 f_2$ is given by
\begin{displaymath}
\begin{bmatrix}
y & -1 & z \\
. & x & -1 \\
. & . & 1 
\end{bmatrix}
s =
\begin{bmatrix}
. \\ . \\ 1
\end{bmatrix},
\quad
s =
\begin{bmatrix}
y\inv(x\inv - z) \\
x\inv \\
1
\end{bmatrix}.
\end{displaymath}
The reason is that the rank does not increase
(when $f_2$ is multiplied by $x\inv$ and $y\inv$ from the left),
because $1 \in R(x\inv f_2)$:
\begin{displaymath}
\begin{bmatrix}
1 & . & . 
\end{bmatrix}
=
t
\begin{bmatrix}
x & -x & -1 \\
. & 1 & z \\
. & . & 1
\end{bmatrix},
\quad
t = 
\begin{bmatrix}
x\inv & 1 & x\inv - z
\end{bmatrix}.
\end{displaymath}
By adding column~2 to column~1 and switching the first two rows
(this results in switching the first two columns in $t$)
we get $[1, 0, 0]\trp$ as the first column in the system matrix
(for the existence of these transformations see Lemma~\ref{lem:ft.for2}):
\begin{displaymath}
\begin{bmatrix}
1 & . & . 
\end{bmatrix}
=
t
\begin{bmatrix}
1 & 1 & z \\
0 & -x & -1 \\
. & . & 1
\end{bmatrix},
\quad
t = 
\begin{bmatrix}
1 & x\inv & x\inv - z
\end{bmatrix}.
\end{displaymath}
After multiplying $x\inv f_2$ from the left by $y\inv$, we have
$1 \not\in R(f')$. Hence a further multiplication by (for example)
$z\inv$ from the left \emph{increases} the rank: 
\begin{displaymath}
\begin{bmatrix}
z & -1 & . & . \\
. & y & -1 & z \\
. & . & x & -1 \\
. & . & . & 1 
\end{bmatrix}
s =
\begin{bmatrix}
. \\ . \\ . \\ 1
\end{bmatrix},
\quad
s =
\begin{bmatrix}
z\inv f' \\
f' \\
x\inv \\
1
\end{bmatrix}.
\end{displaymath}
To summarize, there are essentially two different factorizations
(on the level of outer factors in Definition~\ref{def:ft.factors})
of $f = (xy)\inv (1-xz) (yz)\inv$, namely
\begin{displaymath}
f = \Bigl( y\inv \bigl( x\inv (1-xz) \bigr) \Bigr) (yz)\inv
  = (xy)\inv \Bigl( \bigl( (1-xz) z\inv \bigr) y\inv \Bigr).
\end{displaymath}
\end{Example}

Now we come to the main definition which will
generalize that of left and right divisors
in the free associative algebra (Definition~\ref{def:ft.lrcop}).
To be able to show in Theorem~\ref{thr:ft.ldivs}
that these definitions are indeed
equivalent on $\freeALG{\field{K}}{X}$,
some preparation is necessary.

\begin{figure}
\begin{center}
\includegraphics{als.301}
\end{center}
\caption{Four different derivation trees
  $\tau_k \in \mathcal{T}(f;f_1,f_2,\ldots,f_5)$ with their
  respective (non-trivial) subtrees
  $\tau_k(\text{l})$ and $\tau_k(\text{r})$.
  The factorizations
  \ifJOURNAL\else, they induce,\fi are
  $f_{\tau_1} = (f_1 f_2)\bigl( (f_3 f_4) f_5 \bigr)$,
  $f_{\tau_2} = (f_1 f_2)\bigl( f_3 (f_4 f_5) \bigr)$,
  $f_{\tau_3} = \bigl((f_1 f_2) f_3 \bigr) (f_4 f_5)$
  resp.~$f_{\tau_4} = \bigl(f_1 (f_2 f_3) \bigr) (f_4 f_5)$.
  }
\label{fig:ft.trees}
\end{figure}

\begin{notation}
Let $m \ge 2$ and $f = f_1 f_2 \cdots f_m$ be a product of
$m$ elements $f_i \in \field{F}$. By $\mathcal{T} = \mathcal{T}(f;f_1,f_2,\ldots,f_m)$
we denote the set of ``multiplicative'' \emph{derivation trees} 
\cite[Section~4.1]{Rigo2016a}
\ (or \emph{parse trees}), that is,
\emph{complete plane binary trees} rooted at $f$ with $m$ leaves
$f_1,f_2,\ldots,f_m$
\cite{Stanley2012a}
\ (or \emph{ordered binary trees}).
Now we fix some $\tau_0 \in \mathcal{T}$
and call a subtree $\tau$ of $\tau_0$ \emph{non-trivial} if
it has at least two leaves.
In this case we denote by 
$\tau(\text{l})$ (respectively $\tau(\text{r})$)
the left (respectively right) subtree of $\tau$
and write $f_{\tau}$ for the product constructed by $\tau$,
that is, $f_{\tau} = f_{\tau(\text{l})} f_{\tau(\text{r})}$
(illustrated in Figure~\ref{fig:ft.trees}).
The \emph{height} (or \emph{length}) of a subtree $\tau$
is defined as $\height\tau = 0$ in the trivial case and
recursively as
$\height\tau = 1 + \max\{ \height\tau(\text{l}), \height\tau(\text{r}) \}$.
\end{notation}

\begin{remark}
In a \emph{complete plane binary tree},
no leaf is left out (that is, even ``trivial'' leaves or subtrees appear)
and subtrees do not ``cross'':
Say $m=5$ (as in Figure~\ref{fig:ft.trees}), $f_2 = 1$ and $f_4 = f_3\inv$.
Then $f = (f_1 f_2) f_5$ is \emph{not} ``complete'',
$f = (f_1 f_3) \bigl( (f_2 f_4) f_5 \bigr)$ is \emph{not} ``plane'' (or ``ordered'')
and $f = (f_1 f_2) (f_3 f_4 f_5)$ is \emph{not} ``binary''.
\end{remark}

\begin{definition}[Left and Right Divisors and Coprime Elements]\label{def:ft.divisors}
Let $\Fbullet = \field{F}^\bullet$.
An element $g\in \Fbullet$ \emph{left divides} $f \in \Fbullet$,
written as $g \Fldivs f$, if, for some $m' < m \in \numN$,
there exist $f_1, f_2, \ldots, f_m \in \Fbullet$
and $\tau_0 \in \mathcal{T}(f;f_1,f_2,\ldots,f_m)$
such that $g = f_1 f_2 \cdots f_{m'}$ and $f = g f_{m'+1} \cdots f_m$
and
$f_{\tau(\text{l})}$ is a \emph{left factor} of
$f_{\tau}=f_{\tau(\text{l})} f_{\tau(\text{r})}$
for all non-trivial subtrees $\tau$ of $\tau_0$.

Two elements $f,g \in \Fbullet$ are called \emph{left coprime}
(in $\Fbullet$)
if for all $h \in \Fbullet$ such that $h\Fldivs f$ and $h\Fldivs g$
implies $h \in \field{K}^\times$,
that is, $h$ is an element of the \emph{trivial group of units}.
Right division $f \Frdivs g$ and the notion of
\emph{right coprime} (in $\Fbullet$) is defined in a similar way.
Two elements (in $\Fbullet$) are called \emph{coprime}
if they are left and right coprime.
\end{definition}


\begin{lemma}[Rank Lemma]\label{lem:ft.rank}
Let $0\neq p,q \in \field{F}$.
Then
\begin{itemize}
\item[(i)] $\rank(pq) = \rank p + \rank q - 1$
  if $p,q \in \freeALG{\field{K}}{X}$,
\item[(ii)] $\rank p = \rank(p\inv) + 1$
  if $p \in \freeALG{\field{K}}{X}$,
\item[(iii)] $\rank(p\inv q) \le \rank(p\inv) + \rank q -  1$
  if $1 \in R(q)$,
\item[(iv)] $\rank(p\inv q) \ge \rank(p\inv) + \rank q -  1$
  if $p\inv$ is a \emph{left factor} of $p\inv q$,
\item[(v)] $\rank(p q\inv) \le \rank p + \rank(q\inv) - 1$
  if $1 \in L(p)$ and
\item[(vi)] $\rank(p q\inv) \ge \rank p + \rank(q\inv) -  1$
  if $q\inv$ is a \emph{right factor} of $p q\inv$.
\end{itemize}
\end{lemma}

\begin{proof}
The rank identities (i) and (ii) are immediate consequences
of the minimal polynomial multiplication
(Proposition~\ref{pro:ft.minmul})
and the minimal inverse (Theorem~\ref{thr:ft.mininv})
respectively.
The inequalities (iv) and (vi) follow directly from
the Definition~\ref{def:ft.factors}.
To prove (iii),
let $p\inv$ and $q$ be given by the \emph{minimal}
admissible linear systems $\als{A}'_p = (u'_p, A'_p, v'_p)$
and $\als{A}_q = (u_q, A_q, v_q)$ of dimension $n'_p$ and $n_q$
respectively.
The construction of Proposition~\ref{pro:ft.mul1}
yields an ALS of dimension $n'$ (for $p\inv q$),
hence 
$\rank(p\inv q) \le n' = \rank(p\inv) + \rank q - 1$.
The proofs of (v) and (iii) are similar.
\end{proof}

\begin{remark}
Note that (iii) and (v) hold in particular for polynomials.
Further, recall from Example~\ref{ex:ft.factors}
that for (iv) to hold in the case of
$p,q \in \freeALG{\field{K}}{X}$
it is necessary but \emph{not} sufficient
that $p$ and $q$ are \emph{left coprime}.
\end{remark}

To illustrate the idea of the following lemma we take
$p = x y z$ and an arbitrary $f$ with $\rank(f) = 2$.
It is easy to see that $\rank(p f\inv) \ge \rank(p) - \rank(f) = 2$
(with equality for $f=yz$) because otherwise we could construct
an ALS of dimension $\rank(p f\inv) + \rank(f\inv) -1 < \rank(p)$.
Now say that $f=xz$. Then $\rank(p f\inv) = \rank(x y x\inv) = 3$.
However, for another polynomial $q$ we get $\rank(f q) = 2 + \rank(q)$
thus $\rank(p f\inv) + \rank(f q) = 3 + 2 + \rank(q) > 3 + \rank(q) = \rank(pq) + 1$.

What is rather simple in a concrete example,
namely to verify that we cannot ``insert'' non-trivial
units turns out to be very technical since we have
to investigate the left and right families in detail.

\begin{lemma}\label{lem:ft.special}
Let $p\in \freeALG{\field{K}}{X} \setminus \field{K}$
and $f\in \field{F} \setminus \field{K}$ such that
$p_{i_0} p_{i_0+1} \cdots p_m f\inv \not\in \field{K}^\times$
for \emph{all} factorizations $p = p_1 p_2 \cdots p_m$
into atoms and \emph{all} $i_0 \in \{ 1, 2, \ldots, m \}$.
Then $\rank(pf\inv) + \rank(f) > \rank(p) + 1$.
\end{lemma}

\begin{proof}
For a fixed non-scalar polynomial $p$ we consider
factorizations $p=p_1 p_2 \cdots p_m$ into $m$ atoms $p_j$.
For notational simplicity let $p_0 = p_{m+1} = 1$.
Here $p_1, p_2, \ldots, p_m$ always denote atoms.
Let
\begin{displaymath}
r = \min \bigl\{ \rank(p_{i_0}p_{i_0+1} \cdots p_m f\inv) \mid
  p = p_1 p_2 \cdots p_m \text{ and }
  i_0 \in \{ 1,2, \ldots, m\} \bigr\}
\end{displaymath}
and $i_0$ and $p_1 p_2 \cdots p_m$ such that this minimum is attained.
By assumption $h = p_{i_0} p_{i_0+1} \cdots p_m f\inv
  \in \field{F} \setminus \field{K}$
with $\rank(h) = r$. Thus $f = h\inv p_{i_0} p_{i_0+1} \cdots p_m$
with \emph{non-scalar} $h$. According to Theorem~\ref{thr:ft.mininv}
there are four cases:
\begin{itemize}
\item $r \ge 2$ and $\rank(h\inv) = r-1$ for type $(1,1)$,
\item $r \ge 2$ and $\rank(h\inv) = r$ for type $(1,0)$,
\item $r \ge 2$ and $\rank(h\inv) = r$ for type $(0,1)$ and
\item $r \ge 1$ and $\rank(h\inv) = r+1$ for type $(0,0)$.
\end{itemize}
Now fix an arbitrary factorization of $p$ (into atoms $q_i$) and
any $1 < \ell \le m$ and
let $p'=q_1 q_2 \cdots q_{\ell-1}$ and $p''=q_{\ell} q_{\ell+1} \cdots q_m$
with ranks $n'$ and $n''$ respectively.
It is enough to show that
\begin{displaymath}
\rank(p' h) + \rank(h\inv p'')
  > \rank(p' p'') + 1
  = \rank(p') + \rank(p'').
\end{displaymath}
We proceed as follows:
Depending on the four cases we construct
---using Proposition~\ref{pro:ft.mul2} and Proposition~\ref{pro:ft.mul1}---
admissible linear systems
for $p'h$ and $h\inv p''$ respectively and find an upper bound
for the number of rows/columns that can be removed (due to $\field{K}$-linear
dependent entries in their left and right families).

We start by assuming type $(1,1)$. 
For $p'h$ we construct an ALS $\als{A}'$ of dimension $n_1 = n' + r-1$
with the block decomposition (as linear representation according to
Theorem~\ref{thr:ft.cohn99.14})
\begin{displaymath}
\pi' = \left(
\begin{bmatrix}
0 & u' & .
\end{bmatrix},
\begin{bmatrix}
A'_{1,1} & A'_{1,2} & A'_{1,3} \\
. & A'_{2,2} & A'_{2,3} \\
. & . & A'_{3,3}
\end{bmatrix},
\begin{bmatrix}
. \\ v' \\ 0
\end{bmatrix}
\right).
\end{displaymath}
For $h\inv p''$ we construct $\als{A}''$ of dimension $n_2 = n'' + r-2$
with the block decomposition
\begin{displaymath}
\pi'' = \left(
\begin{bmatrix}
0 & u'' & .
\end{bmatrix},
\begin{bmatrix}
A''_{1,1} & A''_{1,2} & A''_{1,3} \\
. & A''_{2,2} & A''_{2,3} \\
. & . & A''_{3,3}
\end{bmatrix},
\begin{bmatrix}
. \\ v'' \\ 0
\end{bmatrix}
\right).
\end{displaymath}
Let $k'_t$ (respectively $k'_s$) be the size of block $A'_{1,1}$
(respectively $A'_{3,3}$) in $\pi'$ and
$k''_t$ (respectively $k''_s$) be the size of block $A''_{1,1}$
(respectively $A''_{3,3}$) in $\pi''$.
Firstly, we write the left and the right family of $h\inv$ in terms
of their respective family of $h$:
Let $(s^h_1, s^h_2, \ldots, s^h_r)$ and $(t^h_1, t^h_2, \ldots, t^h_r)$
be the left and right family respectively of some minimal ALS for $h$.
Then $s_{h\inv} = (1, s^h_{r-1}, \ldots, s^h_2) h\inv$
and $t_{h\inv} = h\inv (t^h_{r-1}, \ldots, t^h_2, 1)$
are the families of a minimal ALS for $h\inv$
constructed by Theorem~\ref{thr:ft.mininv}.
Recall that row/column~$n'$ was eliminated in a system of dimension $n'+r$
to get $\als{A}'$ and 
row/column~$r$ was eliminated in a system of dimension $r-1+n''$
to get $\als{A}''$.
Secondly, we take a closer look at the left families of $\als{A}'$ and
$\als{A}''$. They are (without loss of generality)
\begin{align*}
s' &= (s^{p'}_1 h, s^{p'}_2 h, \ldots, s^{p'}_{n'-1} h, s^h_1, s^h_2, \ldots, s^h_r)
  \quad\text{and} \\
s'' &= (\underbrace{h\inv p'', s^h_{r-1} h\inv p'', \ldots, s^h_2 h\inv p''}_{r-1},
        \underbrace{s^{p''}_2, \ldots, s^{p''}_{n''}}_{n''-1})
\end{align*}
respectively.
The first observation is that $k'_s = 0$, that is,
the left family of $\als{A}'$ is $\field{K}$-linearly
independent because $1 \in R(h)$ and Lemma~\ref{lem:ft.slinin}.
The first $r-1$ and the last $n''-1$ components of $s''$ are
$\field{K}$-linearly independent. At most $r-1$ (linear
combinations of) components in $s''$ can be eliminated.
Hence we have $k''_s \le r-1$. However, we claim that
\begin{displaymath}
k''_s \le r-2.
\end{displaymath}
Assume to the contrary that (the rank of $p''$ is large enough and)
$k''_s = r-1$, that is, the block $A''_{3,3}$ has dimension $k''_s$.
Then \emph{all} to $h\inv$ corresponding (linear combinations of)
components in $s''$ can
be eliminated by the last $n''-1$ \emph{polynomial} entries.
If $r=2$ then $h$ is a polynomial, so we assume $r \ge 3$.
Since the left family $(1,s_{r-1}^h, \ldots, s_2^h)h\inv$ of $h\inv$
is $\field{K}$-linear independent and
\begin{displaymath}
\rank(h\inv p'') < \rank(s_k^h h\inv p'') < \rank(p'') = n''
\end{displaymath}
for all $k \in \{ 2, 3, \ldots, r-1 \}$ it follows that these
$s_k^h$'s are polynomials and therefore $h$ is a polynomial because
it is of type $(1,1)$. Hence $h$ and $p''$ have a non-trivial
(left) greatest common divisor which contradicts the minimality
of $r = \rank(h)$.
Thus $k''_s \le r-2$.
Thirdly, we take a closer look at the right families
\begin{align*}
t' &= (\underbrace{t^{p'}_1, t^{p'}_2, \ldots, t^{p'}_{n'-1}, p' t^h_1}_{n'},
        p' t^h_2, \ldots, p' t^h_r)  \quad\text{and} \\
t'' &= (h\inv t^h_{r-1}, \ldots, h\inv t^h_2, h\inv,
          h\inv t^{p''}_2, \ldots, h\inv t^{p''}_{n''}) \\
    &= h\inv (\underbrace{t^h_{r-1}, \ldots, t^h_2, 1}_{r-1},
          t^{p''}_2, \ldots, t^{p''}_{n''}).
\end{align*}
By assumption, the first $r-1$ and the last $n''$ components in $t''$ are
$\field{K}$-linearly independent.
Since the $t^{p''}_i$'s are polynomials, we can eliminate
at most $k \le r-2$ \emph{polynomial} (linear combinations of)
components in $t''$.
But then the corresponding $k$ (linear combinations of)
components in $t'$ are $\field{K}$-linearly independent of
$(t_1^{p'}, \ldots, t_{n'-1}^{p'})$ because they are of the form
$p'\tilde{t}_i^h$ (see Lemma~\ref{lem:ft.tlinin}
or minimal polynomial multiplication).
By assumption, the first $n'$ and the last $r$ components in
$t'$ are $\field{K}$-linearly independent.
Thus at most $r-1-k$ components can be eliminated.
Hence we have $k'_t + k''_t \le r-1$. However, we claim that
\begin{displaymath}
k'_t + k''_t \le r-2.
\end{displaymath}
Assume to the contrary that $k'_t + k''_t = r-1$. Than
$p't^h_r$ would also ``vanish'', that is,
\begin{displaymath}
p'g = p'h + \sum_{j=2}^{r-1} \beta_j p' t^h_j = \sum_{i=1}^{n'-1} \alpha_i t^{p'}_i
  = q 
\end{displaymath}
with $\rank(q) < \rank(p')$. Hence $g = \kappa q_{\ell-1}\inv \cdots q_{\ell'}\inv$
for some $\ell'$ such that $1 \le \ell' < \ell-1$ and $\kappa \in \field{K}$.
By assumption $h$ is of type $(1,1)$,
but $g$ is of type $(0,0)$ by Theorem~\ref{thr:ft.mininv}.
Therefore
$g = h - h_0$ for some (non-zero) $h_0$ of type $(1,1)$.
But that would contradict $\field{K}$-linear independence
of the right family $(t^h_1, t^h_2, \ldots, t^h_r)$.
Finally, for $h$ of type $(1,1)$, we have
\begin{align*}
\rank(p'h) + \rank(h\inv p'')
  &= n' + r-1 - (k'_s + k'_t) + n'' + r-2 - (k''_s + k''_t) \\
  &\ge n' + n'' + 2r-3 - (r-2) - (r-2) \\
  &= n' + n'' + 1 > \rank(p'p'') + 1.
\end{align*}
If $h$ is of type $(0,0)$, then $h\inv$ is of type
$(1,1)$ and $t''$ is $\field{K}$-linearly independent,
so we can use similar arguments.
If $h$ is of type $(1,0)$ then the systems
$\als{A}'$ and $\als{A}''$ are of dimensions
$n' + r -1$ and $n''+r-1$ respectively.
Their left families are
\begin{align*}
s' &= (s^{p'}_1 h, s^{p'}_2 h, \ldots, s^{p'}_{n'-1} h, s^h_1, s^h_2, \ldots, s^h_r)
  \quad\text{and} \\
s'' &= (\underbrace{h\inv p'', s^h_r h\inv p'', \ldots, s^h_2 h\inv p''}_{r},
        \underbrace{s^{p''}_2, \ldots, s^{p''}_{n''}}_{n''-1}).
\end{align*}
By similar arguments to the case $(1,1)$ ---$s'$ is $\field{K}$-linearly
independent--- we get $k_s'' \le r-1$.
The right families are
\begin{align*}
t' &= (\underbrace{t^{p'}_1, t^{p'}_2, \ldots, t^{p'}_{n'-1}, p' t^h_1}_{n'},
       \underbrace{p' t^h_2, \ldots, p' t^h_r}_{r-1})  \quad\text{and} \\
t'' &= h\inv (\underbrace{t^h_r, \ldots, t^h_2, 1}_{r},
          t^{p''}_2, \ldots, t^{p''}_{n''}).
\end{align*}
Since neither $h$ nor $h\inv$ is a polynomial,
at most $r-2$ components in $t'$ or $t''$ can be eliminated,
that is, $k_t' + k_t'' \le r-2$.
Therefore, for $h$ of type $(1,0)$, we have
\begin{align*}
\rank(p'h) + \rank(h\inv p'')
  &= n' + r-1 - (k'_s + k'_t) + n'' + r-1 - (k''_s + k''_t) \\
  &\ge n' + n'' + 2r-2 - (r-1) - (r-2) \\
  &= n' + n'' + 1 > \rank(p'p'') + 1.
\end{align*}
If $h$ is of type $(0,1)$, then we have, by Lemma~\ref{lem:ft.tlinin},
$\field{K}$-linearly independent right family $t''$.
By similar arguments we have $k_t' \le r-1$ and $k_s'+k_s''\le 2$.
Thus, at the end,
we have shown that
$\rank(p f\inv) + \rank(f) > \rank(p) + 1$.
\end{proof}

\begin{lemma}\label{lem:ft.polydivs}
Let $q \in H = \freeALG{\field{K}}{X}^\bullet$
and $p \in \Fbullet = \field{F}^\bullet$.
Then $p \Fldivs q$ implies $q = ph$ with $p,h \in H$.
\end{lemma}

\begin{proof}
For some $m' < m$,
let $p = f_1 f_2 \cdots f_{m'}$ and $q = p f_{m'+1} \cdots f_m$
and $\tau_0 \in \mathcal{T}(q;f_1,f_2,\ldots,f_m)$ such that $q_{\tau(\text{l})}$
is a left factor of $q_{\tau}=q_{\tau(\text{l})} q_{\tau(\text{r})}$
for all non-trivial subtrees $\tau$ of $\tau_0$.
We have to show that $q_{\tau_k} \in H$ for all
non-trivial subtrees $\tau_k$ of $\tau_0$ with root in height $k$
by induction on $k$ from $0$ to $\height(\tau_0)-1$.
For $k=0$ we have $q_{\tau_0} = q \in H$.
Without loss of generality assume that $q_{\tau_k(\text{r})}$
is a \emph{proper right} factor of
$q_{\tau_k} = q_{\tau_k(\text{l})} q_{\tau_k(\text{r})}$
(in this case $q_{\tau_k(\text{l})}$ is a
\emph{proper left} factor of $q_{\tau_k}$;
nothing has to be shown for trivial factors), that is
\begin{align*}
\rank\bigl(q_{\tau_k} q_{\tau_k(\text{r})}\inv\bigr)
  + \rank\bigl(q_{\tau_k(\text{r})}\bigr) &\le 1 + \rank(q_{\tau_k}) \quad\text{and} \\
\rank\bigl(q_{\tau_k(\text{r})}\inv\bigr)
  + \rank\bigl(q_{\tau_k(\text{r})} q_{\tau_k}\inv\bigr) &\le 1 + \rank(q_{\tau_k}\inv).
\end{align*}
By assumption $q_{\tau_k}$ is a polynomial and thus all factorizations
into atoms have the same length, say $\ell_k$
(which depends on $\tau_k$).
We claim that $q_{\tau_k(\text{r})} = \kappa g_{\ell_0} \cdots g_{\ell_k}$
for some factorization $q_{\tau_k} = g_1 g_2 \cdots g_{\ell_k}$,
some $\ell_0 \in \{ 1, 2, \ldots, \ell_k \}$ and $\kappa \in \field{K}$.
Assume the contrary and apply Lemma~\ref{lem:ft.special}
with $p=q_{\tau_k}$ and $f = q_{\tau_k(\text{r})}$ to get the
contradiction
\begin{displaymath}
\rank\bigl(q_{\tau_k} q_{\tau_k(\text{r})}\inv\bigr)
  + \rank(q_{\tau_k(\text{r})}) > 1 + \rank(q_{\tau_k}).
\end{displaymath}
Thus $q_{\tau_k(\text{r})} \in H$ and
$q_{\tau_k(\text{l})} = q_{\tau_k} q_{\tau_k(\text{r})}\inv \in H$.
In particular $p = f_1 f_2 \cdots f_m' \in H$
and $h = f_{m'+1} f_{m'+2} \cdots f_m \in H$.
\end{proof}

\begin{theorem}\label{thr:ft.ldivs}
Let $p,q \in H = \freeALG{\field{K}}{X}^\bullet$.
Then $p \ldivs q$ (respectively $p \rdivs q$)
if and only if $p \Fldivs q$ (respectively $p \Frdivs q$)
in $\field{F}$.
\end{theorem}

\begin{proof}
Let $p \ldivs q$, that is $q \in pH = \{ p h \mid h \in H \}$.
We show that $p$ is a left factor of $q = ph$.
By Lemma~\ref{lem:ft.rank} (i) we get
$\rank(q) = \rank(p) + \rank (p\inv q) -1$
and by (ii) we get
$\rank(q\inv) + 1 = \rank(p\inv) +1 + \rank(q\inv p)$,
thus $p$ is a left factor of $ph = q$
and therefore ($m'=1$, $m=2$ and a derivation tree $\tau_0$
of height~$1$ in Definition~\ref{def:ft.divisors})
$p \Fldivs q$ (in $\field{F}$).
Conversely,
we have to show that $q \in pH$.
But this follows directly from 
the assumption $p \Fldivs q$
and Lemma~\ref{lem:ft.polydivs}.
\end{proof}

\begin{notation}
Since the left (respectively right) division in $\freeALG{\field{K}}{X}$
is the same as in $\field{F}$ we can simplify notation
and use $f \ldivs g$ instead of $f \Fldivs g$
(respectively $f\rdivs g$ instead of $f \Frdivs g$)
in the following.
\end{notation}

\begin{definition}[Atoms, Irreducible Elements]\label{def:ft.atoms1}
Let $\Fbullet = \field{F}^\bullet$.
An element $f\in \Fbullet \setminus \field{K}$, that is,
a non-trivial unit (in $\field{F}$), is called (generalized) \emph{atom}
(or \emph{irreducible})
if $f = g_1 g_2$ with $g_1,g_2 \in \Fbullet$ and $g_1 \ldivs f$ implies that
either $g_1 \in \field{K}^\times$ or $g_2 \in \field{K}^\times$.
Like in Definition~\ref{def:ft.atoms},
the set of atoms in $\field{F}$ is denoted by
$\atoms(\field{F})$.
\end{definition}

\begin{remark}
Even in ``simple'' cases it is difficult to decide whether
an element is irreducible (in the general sense of Definition~\ref{def:ft.atoms1})
based on rational expressions.
As an example take $f=1-xy$ and $g=(1-zy)\inv$. Then $f g$ is
\emph{irreducible} while  $gf$ is \emph{reducible}.
One has
to look on their \emph{minimal} linear representations.
Minimal admissible linear systems for $fg$ and $gf$ are
\begin{displaymath}
\begin{bmatrix}
1 & -1 & -x \\
. & y & 1 \\
. & 1 & z
\end{bmatrix}
s =
\begin{bmatrix}
. \\ . \\ 1
\end{bmatrix}
\quad\text{and}\quad
\begin{bmatrix}
y & 1 & . & . \\
1 & z & -x & -1 \\
. & . & 1 & y \\
. & . & . & 1
\end{bmatrix}
s = 
\begin{bmatrix}
. \\ . \\ . \\ 1
\end{bmatrix}
\end{displaymath}
respectively. 
\end{remark}

\begin{Remark}
In other words, a generalized atom can be created \emph{multiplicatively}
out of two atoms. This phenomenon is somewhat curious.
As an example we consider $f = f_1 f_2\inv$ for
$f_1 = 1 - xyz$ and $f_2 = 1- zyz$. A \emph{minimal}
admissible linear system for $f$ is given by
\begin{displaymath}
\begin{bmatrix}
1 & -1 & . & -x \\
. & z & 1 & . \\
. & . & y & -1 \\
. & 1 & . & z
\end{bmatrix}
s =
\begin{bmatrix}
. \\ . \\ . \\ 1
\end{bmatrix}.
\end{displaymath}
Since $\rank(f_1) + \rank(f_2\inv) = 7 > 5 = \rank(f) + 1$,
neither $f_1$ nor $f_2$ is an outer factor (of $f$).
Indeed, $f$ does not have any (non-trivial) outer factors.
\end{Remark}

\begin{remark}
Note the additional condition $g_1 \ldivs f$
compared to Definition~\ref{def:ft.atoms}.
It is \emph{crucial}.
Without $f = 1-x = x \cdot (x\inv -1)$ would not
be an atom. (Actually there would not be any atoms
at all.)
However, $x \nldivs (1-x)$
because $\rank(x) + \rank\bigl(x\inv (1-x) \bigr) = 4 > 3 = 1 + \rank(1-x)$.
\end{remark}

\begin{remark}
A reducible polynomial has only ``polynomial'' divisors.
The result of Lemma~\ref{lem:ft.polydivs} is stronger than
the assumptions for the equivalence of divisbility in Theorem~\ref{thr:ft.ldivs},
in which one divisor is a polynomial.
\end{remark}

\begin{proposition}\label{pro:ft.atoms}
A polynomial is an atom if and only if it is a generalized atom.
\end{proposition}

\begin{proof}
We have to show that
$\atoms(\freeALG{\field{K}}{X}) = \atoms(\field{F}) \cap \freeALG{\field{K}}{X}$.
Recall from Lemma~\ref{lem:ft.polydivs} that
for a polynomial $p$ we have
$g_1 \ldivs p$ implies $p = g_1 g_2$ with polynomials $g_1$ and $g_2$.
Now both implications are immediate.
\end{proof}

\begin{notation}
In the following we use ``atom'' as the general term
and ``polynomial atom'' if we want to emphasize that the
atom is an element in the free associative algebra.
\end{notation}

\section{Minimal Multiplication and Factorization}\label{sec:ft.mf}

Before we describe the correspondence of zero (lower left and upper right)
blocks in the system matrix of a minimal ALS and a non-trivial factorization,
we describe a construction of a minimal ALS $\als{A} = (u,A,v) = (1,A,\lambda)$,
$A = (a_{ij})$,
for the product of
two non-zero elements $f,g$ given by \emph{minimal} admissible linear systems,
say of dimension $n_f$ and $n_g$ respectively,
if $f$ is a \emph{left factor} of $fg$.
According to Theorem~\ref{thr:ft.minmul}
there are three cases (see also Figure~\ref{fig:ft.minmul}, page~\pageref{fig:ft.minmul}):
\begin{center}
\begin{tabular}{lccl}
type& lower left zeros & ``coupling'' & upper right zeros \\\hline\tabstrut
$(1,*)$ & $n_g \times (n_f-1)$ 
        & $\exists\, 1 \le i < n_f : a_{i,n_f} \not\in \field{K}$
        & $(n_f-1)\times (n_g-1)$ \\
$(*,1)$ & $(n_g-1)\times n_f$ 
        & $\exists\, 1 \le j < n_g : a_{n_f,n_f+j} \not\in \field{K}$
        & $(n_f-1)\times (n_g-1)$ \\
$(0,0)$ & $n_g \times n_f$ 
        & $\forall\, i =1,\ldots,n_f: a_{i,n_f+1}\in \field{K}$
        & $n_f \times (n_g-1)$ 
\end{tabular}
\end{center}
Note that for type $(1,*)$ and $(*,1)$ the ``coupling condition''
must hold for \emph{each} (admissibly) transformed system
because otherwise both types could be ``derived'' easily from
type $(0,0)$.
To ``reverse'' the multiplication we need to transform an
ALS accordingly using transformations of the form
\begin{equation}\label{eqn:ft.trans}
(P,Q) = \left(
\begin{bmatrix}
\alpha_{1,1} & \ldots & \alpha_{1,n-1} & 0 \\
\vdots & \ddots & \vdots & \vdots \\
\alpha_{n-1,1} & \ldots & \alpha_{n-1,n-1} & 0 \\
\alpha_{n,1} & \ldots & \alpha_{n,n-1} & 1
\end{bmatrix},
\begin{bmatrix}
1 & 0 & \ldots & 0 \\
\beta_{2,1} & \beta_{2,2} & \ldots & \beta_{n,2} \\
\vdots & \vdots & \ddots & \vdots \\
\beta_{n,1} & \beta_{n,2} & \ldots & \beta_{n,n}
\end{bmatrix}
\right).
\end{equation}
with entries $\alpha_{ij}, \beta_{ij} \in \field{K}$.
To ensure invertibility we need $\det P \neq 0$ and $\det Q \neq 0$.

\begin{figure}
\begin{center}
\includegraphics{als.302}
\end{center}
\caption{There are three types of factorization of an element $h=fg$
  with $\rank(h)=n$, $\rank(f)=k$ and $\rank(g) = n-k$ for type $(0,0)$
  or $\rank(g) = n-k+1$ otherwise.
  These types correspond to that of the minimal multiplication.
  For type $(1,*)$ and $(*,1)$ the coupling has to be \emph{non-scalar}
  for \emph{all} transformations yielding appropriate zero blocks.
  The upper left (grey) block corresponds (modulo ``overlapping'')
  to the system matrix $A_f$ of the ALS $\als{A}_f = (u_f,A_f,v_f)$ of $f$,
  the lower right to $A_g$ of the ALS $\als{A}_g = (u_g,A_g,v_g)$ of $g$.
  White blocks denote zeros, black blocks contain at least one
  non-scalar entry.
  }
\label{fig:ft.minmul}
\end{figure}

\begin{remark}
The \emph{minimal polynomial multiplication} (Proposition~\ref{pro:ft.minmul})
can be formulated as a corollary to the following theorem.
The difficulty of the proof of the former is hidden
in the definition of \emph{outer} factors, Definition~\ref{def:ft.factors}.
To test if $f$ is a left factor of $fg$ in general
relies on techniques for minimization of linear representations
which is discussed in 
\cite{Schrempf2018a}
.
\end{remark}

\begin{theorem}[Minimal Multiplication]\label{thr:ft.minmul}
Let $f,g\in \field{F} \setminus \field{K}$ be given by the
\emph{minimal} admissible linear systems
$\als{A}_f = (u_f, A_f, v_f)$ and $\als{A}_g = (u_g, A_g, v_g)$
of dimension $n_f$ and $n_g$ respectively.
Let $n = n_f + n_g$.
If $f$ is a \emph{left factor} of $fg$, then a \emph{minimal} ALS for $fg$
is given by
\begin{displaymath}
\als{A} = 
\begin{cases}
\text{Proposition~\ref{pro:ft.mul2} with $\dim \als{A} = n-1$}
  & \text{if $1\in L(f)$,} \\
\text{Proposition~\ref{pro:ft.mul1} with $\dim \als{A} = n-1$}
  & \text{if $1\in R(g)$,} \\
\text{Proposition~\ref{pro:ft.ratop} with $\dim \als{A} = n$}
  & \text{if $1 \not\in R(g)$ and $1 \not\in L(f)$.}
\end{cases}
\end{displaymath}
\end{theorem}

\begin{proof}
Since $f$ is a left factor of $fg$, we have
$\rank(f) + \rank(g) \le \rank(fg) + 1$, thus
\begin{displaymath}
\rank(fg) \ge \rank(f) + \rank(g) - 1 = \dim \als{A} \ge \rank(fg)
\end{displaymath}
and hence minimality of $\als{A}$ if $1 \in R(g)$ or $1 \in L(f)$,
that is, the first two cases/types $(*,1)$ and $(1,*)$.
For the last case/type $(0,0)$ we distinguish four subcases.
Recall from Theorem~\ref{thr:ft.mininv} that
\begin{displaymath}
\rank(h\inv ) = 
\begin{cases}
\text{$\rank(h) - 1$} & \text{if $h$ is of type $(1,1)$}, \\
\text{$\rank(h)$} & \text{if $h$ is of type $(1,0)$ or $(0,1)$ and} \\
\text{$\rank(h)+1$} & \text{if $h$ is of type $(0,0)$}.
\end{cases}
\end{displaymath}
Since $f$ is a left factor of $fg$, we have also
\begin{displaymath}
\rank(g\inv f\inv) + 1 \ge \rank(g\inv) + \rank(f\inv).
\end{displaymath}
Hence ---by the minimal inverse on the right hand side---
\begin{displaymath}
\rank(g\inv f\inv) +1 \ge
\begin{cases}
\text{$\rank(g) + \rank(f)$} & \text{if $1 \in R(f)$ and $1 \in L(g)$}, \\
\text{$\rank(g) + 1 + \rank(f)$} & \text{if $1 \in R(f)$ and $1 \not\in L(g)$}, \\
\text{$\rank(g) + \rank(f)+1$} & \text{if $1 \not\in R(f)$ and $1 \in L(g)$ and} \\
\text{$\rank(g) +1 + \rank(f)+1$} & \text{if $1 \not\in R(f)$ and $1 \not\in L(g)$.}
\end{cases}
\end{displaymath}
Note that a priori we cannot assume minimality of $\als{A}$ for $fg$,
since this is what we have to prove. Therefore we cannot use the minimal
inverse on the left hand side because we only know that,
for example, $1 \in L(g)$ implies $1 \in L(\als{A})$ by construction.
However, by the minimal inverse, we know
---since $g$ is of type $(0,*)$--- that
$g\inv$ is of type $(1,1)$ or $(0,1)$
and $f\inv$ is of type $(1,1)$ or $(1,0)$.
Hence we can use one of the first two cases
and get $\rank(fg) \ge \rank(f) + \rank(g)$.
\end{proof}

\begin{remark}
Let $A = (a_{ij})$ be the system matrix of the ALS from
Theorem~\ref{thr:ft.minmul}.
For type $(1,*)$ there exists an $i\in \{1,2,\ldots, n_f-1 \}$
such that $a_{i,n_f}$ is \emph{non-scalar}.
For type $(*,1)$ there exists an $j\in \{ n_f+1, n_f+2, \ldots, n \}$
such that $a_{n_f, j}$ is \emph{non-scalar}.
And for type $(0,0)$ the entries $a_{i,n_f+1}$ are \emph{scalar}
for $i \in \{ 1, 2, \ldots, n_f \}$.
We refer to that as \emph{coupling conditions}.
Note that there is no transformation of the form~\eqref{eqn:ft.trans}
respecting the zero blocks yielding a ``scalar coupling''
in type $(1,*)$ (respectively type $(*,1)$)
because that would contradict minimality of $\als{A}_f$
(respectively $\als{A}_g$)
which one can see after recalling the
construction of the product in Proposition~\ref{pro:ft.ratop}.
\end{remark}

\begin{lemma}\label{lem:ft.factorization}
Let $h,f,g \in \field{F} \setminus \field{K}$ be given by the \emph{minimal}
admissible linear systems $\als{A} = (1,A,\lambda)$,
$\als{A}_f = (u_f,A_f,v_f) = (1,A_f,\lambda_f)$ and
$\als{A}_g = (u_g,A_g,v_g) = (1,A_g,\lambda_g)$
of dimension $n$, $n_f$ and $n_g$ respectively
such that $f$ is a \emph{left factor} of $h = fg$.

\medskip
\noindent
Type~$(1,*)$:
If $f$ is of type $(*,1)$ then there exists an
admissible transformation $(P,Q)$ of the form~\eqref{eqn:ft.trans}
such that $PAQ = (a_{i,j})$ has
\begin{itemize}
\item a lower left block of zeros of size $n_g \times (n_f-1)$,
\item an upper right block of zeros of size $(n_f-1) \times (n_g-1)$ and
\item there exists an $i \in \{ 1,2,\ldots, n_f-1 \}$ 
  such that $a_{i,n_f}$ is \emph{non-scalar}
\end{itemize}
Type~$(*,1)$:
If $g$ is of type $(1,*)$ then there exists an
admissible transformation $(P,Q)$ of the form~\eqref{eqn:ft.trans}
such that $PAQ = (a_{i,j})$ has
\begin{itemize}
\item a lower left block of zeros of size $(n_g-1) \times n_f$,
\item an upper right block of zeros of size $(n_f-1) \times (n_g-1)$ and
\item there exists an $j\in \{ n_f+1, n_f+2, \ldots, n \}$
  such that $a_{n_f, j}$ is \emph{non-scalar}.
\end{itemize}
Type~$(0,0)$:
If $f$ is of type $(*,0)$ and $g$ is of type $(0,*)$
then there exists an admissible transformation $(P,Q)$
of the form~\eqref{eqn:ft.trans}
such that $PAQ = (a_{i,j})$ has
\begin{itemize}
\item a lower left block of zeros of size $n_g \times n_f$,
\item an upper right block of zeros of size $n_f \times (n_g-1)$ and
\item $a_{i,n_f+1} \in \field{K}$ for $i \in \{ 1, 2, \ldots, n_f \}$.
\end{itemize}
\end{lemma}

\begin{proof}
Let $\als{A}'=(u',A',v') = (1,A',\lambda')$
be the minimal ALS for $h=fg$ constructed by Theorem~\ref{thr:ft.minmul}
from $\tsfrac{\lambda_g}{\lambda} \als{A}_f$ and
$\tsfrac{\lambda}{\lambda_g} \als{A}_g$.
The system matrix $A'=(a'_{ij})$ has ---by construction---
appropriate (lower left and upper right) blocks of zeros
and ---for type $(0,0)$---
scalar entries $a'_{i,n_f+1}$ for $i \in \{ 1, 2, \ldots, n_f \}$.
Since both systems $\als{A}$ and $\als{A}'$ for $h$ are minimal,
there exists, by Theorem~\ref{thr:ft.cohn99.14}, an admissible transformation
$(P,Q)$ such that $P\als{A}Q = \als{A}'$.
The right hand side $Pv = v'$ does not change, hence
$(P,Q)$ is of the form~\eqref{eqn:ft.trans}.
The coupling conditions are fulfilled due to the construction
of the minimal multiplication.
\end{proof}

\begin{Example}
To illustrate the importance of the coupling conditions
we consider $h = x\inv z y\inv x\inv$ given by the minimal ALS
\begin{displaymath}
\begin{bmatrix}
x & -z & . \\
. & y & -1 \\
. & . & x
\end{bmatrix}
s =
\begin{bmatrix}
. \\ . \\ 1
\end{bmatrix}.
\end{displaymath}
Multiplication of type $(0,1)$ for $n_f = n_g = 2$ would
violate the coupling condition, ``creating'' a \emph{non-minimal}
ALS for $g$ in $h = fg$.
\end{Example}

\begin{lemma}[Factorization Type~$(1,*)$]\label{lem:ft.fact2}
Let $h = fg \in \field{F}\setminus\field{K}$ be given
by the \emph{minimal} admissible linear system $\als{A} = (u,A,v) = (1,A,\lambda)$
of dimension $n \ge 2$ and fix $1 < k \le n$.
Assume that $A$ has
a lower left block of zeros of size $(n-k+1)\times (k-1)$
and an upper right block of zeros of size $(k-1)\times (n-k)$.
For a transformation $(P,Q)$ let $a'_{ij}$ denote the entries of $PAQ$.
If for \emph{each} transformation $(P,Q)$ of the form~\eqref{eqn:ft.trans}
respecting these zero blocks there exists an $i\in \{1,2,\ldots, k-1 \}$
such that $a'_{i,k}$ is \emph{non-scalar}
then $f$ is a \emph{left factor} of type $(*,1)$ of $h$
with $\rank(f) = k$ and $\rank(g) = n-k+1$.
\end{lemma}

\begin{proof}
By assumption, $\als{A}$ is of the (block) form
\begin{displaymath}
\begin{bmatrix}
A_{1,1} & A_{1,2} & . \\
. & A_{2,2} & A_{2,3} \\
. & A_{3,2} & A_{3,3} 
\end{bmatrix}
\begin{bmatrix}
s_{\block{1}} \\ s_k \\ s_{\block{3}}
\end{bmatrix}
=
\begin{bmatrix}
. \\ . \\ v_{\block{3}}
\end{bmatrix}
\end{displaymath}
with square diagonal blocks $A_{1,1}$, $A_{2,2}$ and $A_{3,3}$ of size
$k-1$, $1$ and $n-k-1$ respectively.
We duplicate the entry $s_k$ in the left family by inserting a ``dummy''
row (and column) to get the following ALS of dimension $n+1$:
\begin{displaymath}
\begin{bmatrix}
A_{1,1} & A_{1,2} & 0 & . \\
0 & 1 & -1 & 0 \\
. & 0 & A_{2,2} & A_{2,3} \\
. & 0 & A_{3,2} & A_{3,3} 
\end{bmatrix}
\begin{bmatrix}
s_{\block{1}} \\ s_k \\ s_k \\ s_{\block{3}}
\end{bmatrix}
=
\begin{bmatrix}
. \\ . \\ . \\ v_{\block{3}}
\end{bmatrix},
\end{displaymath}
that is, ``reversing'' the construction from Proposition~\ref{pro:ft.mul2}.
The subsystems of dimension $k$ and $n-k+1$ are \emph{minimal} for
$f$ (due to the coupling condition) and $g = \mu s_k$ respectively,
otherwise we could construct an ALS for $h$ of dimension $n' < n$,
contradicting minimality of $\als{A}$.
Clearly, $1 \in L(f)$.
By construction we have $\rank(f) + \rank(g) = \rank(h) + 1$,
thus we only have to show that
$\rank(g\inv) + \rank(f\inv) \le \rank(h\inv) + 1$
for $f$ to be a left factor of $h$
by distinguishing four cases (like in the minimal multiplication)
and apply the minimal inverse. If $h$ is of type $(1,1)$,
then $f$ is of type $(1,1)$ and $g$ is of type $(*,1)$.
Thus $\rank(g\inv)\le \rank(g)$ and we get
$\rank(g\inv) + \rank(f\inv) \le \rank(g) + \rank(f)-1 = \rank(h\inv)+1$.
The other cases are as easy.
\end{proof}

\begin{lemma}[Factorization Type~$(*,1)$]\label{lem:ft.fact1}
Let $h = fg \in \field{F}\setminus\field{K}$ be given
by the \emph{minimal} admissible linear system $\als{A} = (u,A,v) = (1,A,\lambda)$
of dimension $n \ge 2$ and fix $1 \le k < n$.
Assume that $A$ has
a lower left block of zeros of size $(n-k)\times k$
and an upper right block of zeros of size $(k-1)\times (n-k)$.
For a transformation $(P,Q)$ let $a'_{ij}$ denote the entries of $PAQ$.
If for \emph{each} transformation $(P,Q)$ of the form~\eqref{eqn:ft.trans}
respecting these zero blocks there exists an $j\in \{k+1, k+2, \ldots, n \}$
such that $a'_{k,j}$ is \emph{non-scalar}
then $f$ is a \emph{left factor} of type $(*,1)$ of $h$
with $\rank(f) = k$ and $\rank(g) = n-k+1$.
\end{lemma}

\begin{proof}
By assumption, $\als{A}$ is of the (block) form
\begin{displaymath}
\begin{bmatrix}
u_{\block{1}} & . & . 
\end{bmatrix}
= 
\begin{bmatrix}
t_{\block{1}} & t_k & t_{\block{3}}
\end{bmatrix} 
\begin{bmatrix}
A_{1,1} & A_{1,2} & . \\
A_{2,1} & A_{2,2} & A_{2,3} \\
. & . & A_{3,3}
\end{bmatrix}
\end{displaymath}
with square diagonal blocks $A_{1,1}$, $A_{2,2}$ and $A_{3,3}$
of size $k-1$, $1$ and $n-k$ respectively. We duplicate
the entry $t_k$ in the right family by inserting a ``dummy'' column
(and row) to get the following ALS of dimension $n+1$:
\begin{displaymath}
\begin{bmatrix}
u_{\block{1}} & . & . & . 
\end{bmatrix}
= 
\begin{bmatrix}
t_{\block{1}} & t_k & t_k & t_{\block{3}}
\end{bmatrix} 
\begin{bmatrix}
A_{1,1} & A_{1,2} & 0 & . \\
A_{2,1} & A_{2,2} & -1 & 0 \\
0 & 0 & 1 & A_{2,3} \\
. & . & 0 & A_{3,3}
\end{bmatrix},
\end{displaymath}
that is, ``reversing'' the construction from Proposition~\ref{pro:ft.mul1}.
The subsystems of dimension $k$ and $n-k+1$ are \emph{minimal} for
$f = \mu t_k$ and $g$ (due to the coupling condition) respectively,
otherwise we could construct an ALS for $h$
of dimension $n' < n$, contradicting minimality of $\als{A}$.
Clearly, $1 \in R(g)$. 
Showing that $f$ is a left factor of $h=fg$ is like
in Lemma~\ref{lem:ft.fact2}.
\end{proof}

\begin{lemma}[Factorization Type~$(0,0)$]\label{lem:ft.fact0}
Let $h = fg \in \field{F}\setminus\field{K}$ be given
by the \emph{minimal} admissible linear system $\als{A} = (u,A,v) = (1,A,\lambda)$
of dimension $n \ge 2$ and fix $1 \le k < n$.
If $A = (a_{ij})$ has
a lower left block of zeros of size $(n-k)\times k$,
an upper right block of zeros of size $k\times (n-k-1)$
and $a_{i,k+1} \in \field{K}$ for $i \in \{ 1, 2, \ldots, k \}$
then $f$ is a \emph{left factor} of type $(*,0)$ of $h$ with $\rank(f) = k$
and $g$ is of type $(0,*)$ with $\rank(g) = n-k$.
\end{lemma}

\begin{proof}
We get the subsystems $\als{A}_f$ (for $f$) and $\als{A}_g$ (for $g$)
directly from the construction of the multiplication in Proposition~\ref{pro:ft.ratop}.
Non-minimality of one of them would contradict
minimality of $\als{A}$. As would $1 \in L(f)$ or $1 \in R(g)$
using multiplication type~$(*,1)$ and~$(1,*)$ respectively.
The arguments for showing that $f$ is left factor of $h=fg$
are similar to that in (the proof of) Lemma~\ref{lem:ft.fact2}.
\end{proof}

\begin{theorem}[Free Factorization]\label{thr:ft.factorization}
Let $h \in \field{F}$ with $n = \rank(h) \ge 2$
be given by the \emph{minimal} admissible linear system $\als{A} = (u,A,v)$.
Then $h$ has a proper left factor $f$ with $\rank(f) = k$ if and only if
there exists an admissible transformation $(P,Q)$ of the form~\eqref{eqn:ft.trans}
such that $PAQ$ is of ``type'' $(1,*)$, $(*,1)$ or $(0,0)$
as in Figure~\ref{fig:ft.minmul}, page~\pageref{fig:ft.minmul}.
\end{theorem}

\begin{proof}
Assuming a proper left factor of rank~$k$,
Lemma~\ref{lem:ft.factorization} applies.
Conversely, assuming such a transformation,
we get a proper left factor of rank~$k$
by Lemma~\ref{lem:ft.fact2} for type~$(1,*)$,
by Lemma~\ref{lem:ft.fact1} for type~$(*,1)$ and
by Lemma~\ref{lem:ft.fact0} for type~$(0,0)$.
\end{proof}

Fixing a rank of a possible left factor
in $\freeFLD{\aclo{\field{K}}}{X}$,
a variant of 
\cite[Theorem~4.1]{Cohn1999a}
\ can be used to detect the lower left and upper right block of zeros
(of appropriate sizes depending on the type of factorization).
Notice that there is a misprint, the coefficients corresponding
to $1 \in X^*$ are missing. Here we have
\begin{align*}
\field{K}[\alpha,\beta] = \field{K}[
  &\alpha_{1,1},\ldots,\alpha_{1,n-1},\alpha_{2,1},\ldots,\alpha_{2,n-1},
  \ldots, \alpha_{n,1},\ldots,\alpha_{n,n-1}, \\
  &\beta_{2,1},\ldots,\beta_{2,n},\beta_{3,1},\ldots,\beta_{3,n},
  \ldots,\beta_{n,1},\ldots,\beta_{n,n} ].
\end{align*}
The coupling conditions for type $(0,0)$ have to be implemented
directly by adding the coefficients corresponding
to $x \in X$ for the ``coupling vector''.
For type $(1,*)$ and $(*,1)$ one can test for a
``scalar'' coupling first. If there is no solution
one can try to find an appropriate transformation
for the zero blocks only.

\begin{Example}\label{ex:ft.factorization}
Let the element $f \in \freeFLD{\numQ}{X}$ be given by the
\emph{minimal} ALS $\als{A} = (u,A,v)$,
\begin{displaymath}
\begin{bmatrix}
-1 & . & x & -1 \\
1+x & x & -1 & . \\
y & 1 & x & -1 \\
x & . & -2 & x
\end{bmatrix}
s =
\begin{bmatrix}
. \\ . \\ . \\ 1
\end{bmatrix}.
\end{displaymath}
Before we start a ``brute force'' attack and try to find
(admissible) transformations $(P,Q)$, say for multiplication
type $(0,0)$ and two systems of dimension~2, we can easily
find out that $f$ is \emph{regular}. If it cannot be tranformed
into a polynomial form, that is, $f$ is no polynomial,
we could check if $f\inv \in \freeALG{\numQ}{X}$.

Now we try to find a \emph{left factor} $f_1$ of type $(*,0)$
with rank $n_1 = 2$ 
and a \emph{right factor} $f_2$ of type $(0,*)$ with rank $n_2 = 2$,
that is, minimal multiplication type $(0,0)$.
We need an \emph{invertible} transformation $(P,Q)$ of the form
\eqref{eqn:ft.trans} such that $P A Q = (a'_{i,j})$ has a
$2 \times 2$ lower left and a $2 \times 1$ upper right
block of zeros and $a'_{1,3},a'_{2,3} \in \numQ$.
Additional to $\det P =1$ and $\det Q = 1$ we have
$12 + 6 + 4$~equations. A Gröbner basis for the ideal
generated by these 24~equations (computed by \textsc{FriCAS}
\cite{FRICAS2018}
, using \emph{lexicographic order}) is
\begin{align*}
( &\alpha_{1,1} + \alpha_{1,2}\beta_{2,2}\beta_{3,3}\beta_{3,4} + \alpha_{1,3} , 
  \quad \alpha_{1,2} \alpha_{2,3}\alpha_{3,1} -\alpha_{1,3}\alpha_{2,2}\alpha_{3,1} -1, \\&\alpha_{2,1} + \alpha_{2,2}\beta_{2,2}\beta_{3,3}\beta_{3,4} + \alpha_{2,3},
  \quad \alpha_{3,2}, \quad \alpha_{3,3}, \quad \alpha_{4,2}, \quad \alpha_{4,3}, \\
&\beta_{2,2}\beta_{3,3}^2 \beta_{3,4} - \beta_{2,3},
  \quad \beta_{2,2}\beta_{3,3}\beta_{4,4} - \beta_{2,2}\beta_{3,4}\beta_{4,3} -1,
  \quad \beta_{2,3}^2, \quad \beta_{2,3}\beta_{3,4}, \\
&\beta_{2,3}\beta_{4,4} - \beta_{3,3}\beta_{3,4},
  \quad \beta_{2,4}, \quad \beta_{3,1}, \quad \beta_{3,2}, \quad \beta_{3,4}^2, 
  \quad \beta_{4,1}+1, \quad \beta_{4,2} ).
\end{align*}
Since $\beta_{3,4}=0$, the transformation $(P,Q)$ is of the form
\begin{displaymath}
(P,Q) = \left(
\begin{bmatrix}
\alpha_{1,1} & \alpha_{1,2} & -\alpha_{1,1} & . \\
\alpha_{2,1} & \alpha_{2,2} & -\alpha_{2,1} & . \\
\alpha_{3,1} & 0 & 0 & . \\
\alpha_{4,1} & 0 & 0 & 1 \\
\end{bmatrix},
\begin{bmatrix}
1 & . & . & . \\
\beta_{2,1} & \beta_{2,2} & 0 & 0 \\
0 & 0 & \beta_{3,3} & 0 \\
-1 & 0 & \beta_{4,3} & \beta_{4,4} \\
\end{bmatrix}
\right)
\end{displaymath}
with a solution over $\numQ$:
\begin{displaymath}
(P,Q) = \left(
\begin{bmatrix}
2 & 0 & -2 & . \\
0 & 1 & 0 & . \\
\frac{1}{2} & 0 & 0 & . \\
0 & 0 & 0 & 1
\end{bmatrix},
\begin{bmatrix}
1 & . & . & . \\
0 & 1 & 0 & 0 \\
0 & 0 & 1 & 0 \\
-1 & 0 & 0 & 1
\end{bmatrix}
\right).
\end{displaymath}
The transformed system $P \als{A} Q$ is
\begin{displaymath}
\begin{bmatrix}
-2 -2y & -2 & . & 0 \\
1+x & x & -1 & 0 \\
0 & 0 & \frac{1}{2} x & -\frac{1}{2} \\
0 & 0 & -2 & x 
\end{bmatrix}
s =
\begin{bmatrix}
. \\ . \\ . \\ 1
\end{bmatrix},
\end{displaymath}
that is, $f = f_1 f_2$, with
$f_1 = (1-xy)\inv$ and $f_2 =  (x^2 - 2)\inv$,
which can be seen easily after applying the minimal
inverse on the two subsystems of dimension $n_1 = n_2 = 2$.
Both factors $f_1,f_2$ are atoms.
Over $\freeFLD{\numC}{X}$ the second factor $f_2$
is \emph{reducible}, 
we have $f_2 = (x - \sqrt{2})\inv (x + \sqrt{2})\inv$.
\end{Example}

\begin{remark}
To find a solution in general (more systematically),
the primary decomposition of ideals can be used,
see for example
\cite[Section~4.8]{Cox2015a}
\ and
\cite[Section~10.8]{Cohn2003a}
.
\end{remark}

\section*{Epilogue}

The presented ``free factorization theory''
is concrete enough to be implemented
in computer algebra software
to be able to apply it.
But some more theoretical questions remain open:
Is the extension of the ``classical'' factorization theory
(in free associative algebras) to the free field ---assuming that
polynomial atoms (and their inverse) remain irreducible---
unique?
Is the free field (in this setting) a ``similarity UFD''?
If so, given an element,
is the sequence of the ranks of the atoms of a factorization
an invariant (modulo permutations)?

\ifJOURNAL
\else
\section*{Acknowledgement}

I thank Daniel Smertnig for the fruitful discussions about
non-commutative factorization and Michael Moßhammer for some
hints on graphs and trees and use this opportunity to
thank Sergey Berezin and Vladimir Vasilchuk
for their support in St.~Petersburg in May 2017.
I am very grateful for the constructive feedback of the anonymous
referees to increase readability,
in particular for the suggested simplification of the definition
of left/right divisibility.
\fi

\ifJOURNAL
\bibliographystyle{plain}
\else
\bibliographystyle{alpha}
\fi
\bibliography{doku}
\addcontentsline{toc}{section}{Bibliography}

\end{document}